\documentclass[preprint]{imsart}

\usepackage{natbib}
\bibpunct{(}{)}{,}{a}{}{;}

\usepackage[utf8]{inputenc}
\usepackage[english]{babel}
\usepackage{amsmath,amsthm,amssymb,amsfonts,epic,latexsym,hyperref
}
\usepackage{mathabx}
\usepackage[dvips]{graphicx}
\usepackage[usenames]{color}

\usepackage{lipsum}
\usepackage{rotating}
\usepackage{dsfont}
\usepackage{geometry}
\usepackage{stmaryrd}
\usepackage{ifthen}
\usepackage[nottoc]{tocbibind}
\usepackage{verbatim}
\usepackage{tabularx}
\usepackage{framed}
\usepackage{longtable}
\usepackage{tabu}
\usepackage{mathtools}
\usepackage{tikz}
\usepackage[usenames]{color}

\allowdisplaybreaks

\newcommand\un{\mathds{1}}
\newcommand\eps{\varepsilon}
\renewcommand\phi{\varphi}
\newcommand\pro[1]{P\left(#1\right)}
\newcommand\esp[1]{\mathbb{E}\left[#1\right]}
\newcommand\proc[2]{\mathbb{P}\left(\left.#1\right|#2\right)}

\newcommand\uno[1]{\un_{\left\{#1\right\}}}
\newcommand\1{\un}

\newtheorem{thm}{Theorem}[section]
\newtheorem{prop}[thm]{Proposition}

\newtheorem{assu}{Assumption}
\newtheorem{lem}[thm]{Lemma}
\newtheorem{rem}[thm]{Remark}
\newtheorem{defi}[thm]{Definition}
\newtheorem{ex}{Example}

\newcommand\n{\mathbb{N}}
\newcommand\W{\mathcal{W}}
\newcommand\F{\mathcal {F}}
\newcommand\G{\mathcal {G}}
\newcommand\PP{\mathcal {P}}
\newcommand\B{\mathcal {B}}
\renewcommand\r{\mathbb{R}}
\renewcommand\ll\left
\newcommand\rr\right

\begin{document}

\begin{frontmatter}
\title {White-noise driven conditional McKean-Vlasov limits for systems of particles with simultaneous and random jumps}


\runtitle{Conditional McKean-Vlasov limit}

\begin{aug}

\author{\fnms{Xavier} \snm{Erny}\thanksref{m4}\ead[label=e4]{xavier.erny@univ-evry.fr}},
\author{\fnms{Eva}
\snm{L\"ocherbach}\thanksref{m2}\ead[label=e5]{locherbach70@gmail.com}}
\and \author{\fnms{Dasha} \snm{Loukianova}\thanksref{m4} \ead[label=e4]{dasha.loukianova@univ-evry.fr}}

\address{
  \thanksmark{m4}Universit\'e Paris-Saclay, CNRS, Univ Evry, Laboratoire de Math\'ematiques et Mod\'elisation d'Evry, 91037, Evry, France\\
  \thanksmark{m2}Statistique, Analyse et Mod\'elisation Multidisciplinaire, Universit\'e Paris 1 Panth\'eon-Sorbonne, EA 4543 et FR FP2M 2036 CNRS}

\runauthor{X. Erny et al.}

 \end{aug}

\begin{abstract}
We study the convergence of $N-$particle systems described by SDEs driven by Brownian motion and Poisson random measure, where the coefficients depend on the empirical measure of the system. Every particle jumps with a jump rate depending on its position and on the empirical measure of the system. Jumps are simultaneous, that is, at each jump time, all particles of the system are affected by this jump and receive a random jump height that is centred  and scaled in $N^{-1/2}.$ This particular scaling implies that the limit of the empirical measures of the system is random, describing the conditional distribution of one particle in the limit system. We call  such limits {\it conditional McKean-Vlasov limits}. The conditioning in the limit measure reflects the dependencies between coexisting particles in the limit system such that we are dealing with a {\it conditional propagation of chaos property}. As a consequence of the scaling in $N^{-1/2}$ and of the fact that the limit of the empirical measures is not deterministic  the limit system turns out to be solution of a non-linear SDE, where not independent martingale measures and white noises appear having an intensity that depends on the conditional law of the process.
\end{abstract}



\end{frontmatter}

\section{Introduction}

McKean-Vlasov equations are  stochastic differential equations  where the coefficients depend on the distribution of the solution. Such equations
 typically arise as 
limits of mean field $N-$particle systems, where the coefficients depend on the empirical measure of the system. These kind of limits  are referred to as {\it McKean-Vlasov limits}  (see e.g. \cite{gaertner},  \cite{graham_mckean-vlasov_1992} and \cite{andreis_mckeanvlasov_2018}).

In this paper we  
extend these limits 
to a rather general class of {\it conditional McKean-Vlasov limits}. More precisely, we consider the system of interacting particles with a diffusive term and jumps, given by
\begin{multline}
dX^{N,i}_t=b(X^{N,i}_t,\mu^N_t)dt+\sigma(X^{N,i}_t,\mu^N_t)d\beta^i_t\label{XNi}\\
+\frac{1}{\sqrt{N}}\sum_{k=1, k \neq i }^N\int_{\r_+\times E}\Psi(X^{N,k}_{t-},X^{N,i}_{t-},\mu^N_{t-},u^k,u^i)\uno{z\leq f(X^{N,k}_{t-},\mu^N_{t-})}d\pi^k(t,z,u), 1 \le i \le N,\\
\end{multline}
starting from the initial condition 
$$(X^{N,i}_0)_{1\leq i\leq N}\sim\nu_0^{\otimes N}.$$
Here $\mu^N_t=N^{-1}\sum_{j=1}^N\delta_{X^{N,j}_t}$ is the empirical measure of the system, $\beta^i$ ($i\geq 1$) are i.i.d. one-dimensional standard Brownian motions and $\pi^k$ ($k\geq 1$) i.i.d. Poisson measures on $ \r_+ \times \r_+ \times E,$ where $E=\r^{\n^*},$ $ \n^* = \{ 1, 2, 3, \ldots \} .$ Each $\pi^k $ has intensity $ds\cdot dz\cdot \nu(du)$, with $\nu$ a product probability measure on $E.$   The initial distribution $\nu_0$ is a probability measure  on $\r$ having a finite second moment. We assume that the Poisson measures, the Brownian motions and the initial conditions are independent. Since the jumps are scaled in $ N^{-1/2},$  to prevent  the jump term from exploding, we suppose that 
the height of the jump term is centred  (see Assumption~\ref{hyppsi} below), that is,  for all $x,y\in\r,$ for all probability measures $m$ on $\r,$
$$\int_E \Psi(x,y,m,u^1,u^2)d\nu(u)=0.$$

Our model is close to the one considered in \cite{andreis_mckeanvlasov_2018}.  As there, any particle in position $x$ jumps at a rate $ f ( x, m ), $ whenever $ m $ is the common state of the system, that is, the current value of the empirical measure. Mainly motivated by applications coming from neuroscience (where jumps are spikes of the neurons leading to an increase of the potential of all other neurons, see e.g. \cite{de_masi_hydrodynamic_2015} and \cite{fournier_toy_2016}, or \cite{duarte_hydrodynamic_2015} for a spatially structured model),  jumps are simultaneous, that is, all particles in the system are affected by any of the jumps. More precisely, 
the random jump height depends both on the current position $y$ of the particle receiving the jump and on the position $x$ of the particle that causes the jump. Notice that contrarily to \cite{andreis_mckeanvlasov_2018} 
we  do not include auto-interactions induced by jumps in the equation  \eqref{XNi},  i.e. terms of the type 
$$\int_{\r_+\times E}  \Theta (X^{N,i}_{t-},\mu^N_{t-},u^i)\uno{z\leq f(X^{N,i}_{t-},\mu^N_{t-})}d\pi^i(t,z,u).$$
 Indeed, such terms would survive in the large population limit leading to discontinuous trajectories, and the presence of the indicator $\uno{z\leq f(X^{N,i}_{t-},\mu^N_{t-})}$ requires to work both in $L^1 $ and in $L^2 $ (see \cite{graham_mckean-vlasov_1992}, see also \cite{erny_conditional_2020} where we dealt both with simultaneous small jumps and big ones). In the present paper, we decided to disregard these big  jumps to focus on the very specific form of the limit process given in~\eqref{barXi} below.

Coming back to \cite{andreis_mckeanvlasov_2018}, the main difference to our work is that there the averaging regime is considered:
the common contribution of  all particles to the dynamic of a given particle, represented  in \eqref {XNi} by the  sum of the stochastic integrals with respect to the Poisson random measures,
is scaled in $N^{-1}.$
In this situation it was shown in \cite{andreis_mckeanvlasov_2018}   that the {\it Propagation of chaos } phenomenon holds: the coordinates are i.i.d. in the limit. Moreover,  the limit of the empirical measures is the distribution of any coordinate 
of the limit system, and  
the dynamic of one coordinate is described by a classical McKean-Vlasov equation.

The novelty of the present paper is that we consider \eqref{XNi} in a diffusive regime, where the  common contribution of  all particles to the dynamic of a given particle
is scaled in $ N^{- 1/2}. $ 
It has already been observed that this diffusive scaling gives rise to the {\it Conditional propagation of chaos} property  (see \cite{erny_conditional_2020}): a common noise appears in the limit system, and the coordinates of the limit system are conditionally i.i.d given this common noise. Moreover  the limit of the empirical measures is shown to be the conditional distribution of any coordinate of the limit system, given the common noise.  In \cite{erny_conditional_2020} the common noise is a Brownian motion  created by  the contribution of the jumps of all particles in the dynamic of a given particle, as a consequence of  the scaling  $1/\sqrt N$  and the central limit theorem.

It turns out that  in the present work,  to describe the precise dynamic of the limit, and in particular to identify the common noise, we need to rely on martingale measures and white noises 
 (see \cite{walsh_introduction_1986} and \cite{el_karoui_martingale_1990}) as driving measures. More precisely, the limit system will be shown to be solution of a non-linear SDE driven by (white noise) martingale measures  having an intensity that depends on the conditional law of the system itself.  These martingale measures do only appear in the limit system as a consequence of the central limit theorem and the joint contribution of all small and centered jumps.  The main reason for the appearance of the martingale measures instead of the Brownian motion is the spatial correlation of the finite system, i.e. the dependence on the positions both of the particle giving and the one receiving the input.  

To the best of our knowledge, this is the first time that McKean-Vlasov limits are considered where the underlying driving martingale measures are only present in the limit system, but not at the level of the $N-$particle system. We refer however to  \cite{julien-ost} who work in the averaging regime and study the fluctuations of  a stochastic system, associated to spatially structured Hawkes processes, around its mean field limit, and where particles in the mean field limit are still  independent.  

Processes driven by martingale measures having an intensity that depends on the law of the process itself have already appeared in the literature related to particle approximations of Boltzmann's equation, starting with the classical article by  \cite{tanaka} that gave rise to a huge literature (to cite juste a few, see  \cite{graham1997}, \cite{sylvie}, \cite{nicolassylvie},  \cite{fournier2016}). In these papers, the underlying random measure is Poisson, and the dependence on the law arises  at the level of the particle system that is designed to approximate Boltzmann's equation. In our work, the underlying random measure is white noise since jumps disappear in the limit, and the dependence on the (conditional) law of the process does only appear in the limit, as an effect of the conditional propagation of chaos.

 Let us now describe the limit system associated to \eqref{XNi}. To find its precise form, we  mainly need to understand the limits of 
the martingales which are the jump terms of the system,
given by
$$J^{N,i}_t=\frac{1}{\sqrt{N}}\sum_{k=1, k \neq i }^N\int_{[0,t]\times\r_+\times E}\Psi(X^{N,k}_{s-},X^{N,i}_{s-},\mu^N_{s-},u^k,u^i)\uno{z\leq f(X^{N,k}_{s-},\mu^N_{s-})}d\pi^k(s,z,u).$$

In what follows we consider some concrete examples of $\Psi$ and give the limits of the corresponding predictable quadratic covariations of  $J^{N,i}_t$ to have a better understanding of its limit.   We shall always assume that the jump rate function $f$ is bounded. Let us begin with a situation close to that of \cite{erny_conditional_2020}, where in the limit system each coordinate shares a common Brownian motion~$W$.

\begin{ex}\label{commonW}
Suppose that $\Psi(x,y,m,u^1,u^2 )=\Psi(u^1).$   Then we have, for all $1 \leq i,j\leq N,$
\begin{multline*}
\langle J^{N,i},J^{N,j}\rangle_t=\frac1N\sum_{k=1, k \neq i, j }^N\int_0^t\int_\r \Psi(u^k)^2f(X^{N,k}_s,\mu^N_s)d\nu_1(u^k)ds\\
= \varsigma^2\int_0^t \int_\r f(x,\mu^N_s)\mu^N_s(dx)ds + O (\frac{ t}{N}),
\end{multline*}
since $f$ is bounded, with $\varsigma^2:=\int_\r\Psi(u^1)^2d\nu_1(u^1)$ and $\nu_1$ the projection of $\nu$ on the first coordinate. Denote  $\mu$ the limit of the empirical measures $\mu^N.$ Then the angle brackets process should converge as $N$ goes to infinity to
$$\varsigma^2\int_0^t\int_\r f(x,\mu_s)\mu_s(dx)ds.$$
As the  limit quadratic covariations    are non-null, in the limit system  there will be a common Brownian motion $W$ underlying each particle's motion. Thus, the limit system is given by
$$d\bar X^i_t=b(\bar X^{i}_t,\mu_t)dt+\sigma(\bar X^{i}_t,\mu_t)d\beta^i_t+\varsigma\sqrt{\int_\r f(x,\mu_t)\mu_t(dx)}dW_t,$$
where $W$ is a standard one-dimensional Brownian motion. We will also show that $\mu=\mathcal{L}(\bar X^1|W),$ since $\mu$ is necessarily the directing measure of $(\bar X^i)_{i\geq 1}.$ In particular, the conditioning in $\mu$ reflects the presence of some common noise, which is $W$ here.
\end{ex}

Now, let us consider an opposite situation where, in the limit system, each coordinate has its own Brownian motion~$W^i$, and where these Brownian motions are independent. 

\begin{ex}\label{iidW}
In this example, we assume that $\Psi(x,y,m,u^1,u^2)=\Psi(u^2).$ As in the previous example, we begin by computing the angle brackets of the jump terms between particles $i$ and $j$. Here we distinguish two cases: $i\neq j$ and $i=j$. If $i\neq j,$ using the fact that  $ \nu$ is a product measure and that $\Psi$ is centered (Assumption \ref{hyppsi}),
$$
\langle J^{N,i},J^{N,j}\rangle_t=
\frac1N\sum_{k=1, k \neq i, j }^N\int_0^t\int_E \Psi(u^i)\Psi(u^j)f(X^{N,k}_s,\mu^N_s)d\nu(u)ds= 0.
$$
Moreover, if $i=j,$
$$\langle J^{N,i}\rangle_t=\frac1N\sum_{k=1, k \neq i }^N\int_0^t\int_\r \Psi(u^i)^2f(X^{N,k}_s,\mu^N_s)d\nu_1(u^i)ds= \varsigma^2\int_0^t \int_\r f(x,\mu^N_s)\mu^N_s(dx)ds + O (\frac{ t}{N}).$$

As the quadratic  covariations   between different particles are null, there will be no common noise in the limit system. So, instead of having one common Brownian motion $W$ as in the previous example, here, each particle is driven by its own Brownian motion. More precisely, in this example  the limit system is
$$d\bar X^i_t=b(\bar X^{i}_t,\mu_t)dt+\sigma(\bar X^{i}_t,\mu_t)d\beta^i_t+\varsigma\sqrt{\int_\r f(x,\mu_t)\mu_t(dx)}dW^i_t,$$
where $W^i$ ($i\geq 1$) are independent standard one-dimensional Brownian motions, independent of $ \beta^i $ ($i \geq 1 $), and where $\mu=\mathcal{L}(\bar X^1)$ is deterministic in this particular case.
\end{ex}
	
Finally let us show an example where, as in Example~\ref{commonW}, each particle shares a common Brownian motion $W$, and, as in Example~\ref{iidW}, each particle has also its own Brownian motion $W^i,$ and where both $W$ and $W^i $ are produced by the  common contribution of the  small jumps.

\begin{ex}\label{combine}
Here we assume that $\Psi(x,y,m,u^1,u^2)=\Psi(u^1,u^2).$ The angle brackets of the jump terms of the particles $i$ and $j$ are, if $i\neq j,$
\begin{multline*}
\langle J^{N,i},J^{N,j}\rangle_t=\frac1N\sum_{k=1, k \neq i, j }^N\int_0^t\int_E \Psi(u^k,u^i)\Psi(u^k,u^j)f(X^{N,k}_s,\mu^N_s)d\nu(u)ds\\
={ \xi^2} \int_0^t \int_\r f(x,\mu^N_s)\mu^N_s(dx)ds + O (\frac{ t}{N}),
\end{multline*}
where we know that ${\xi^2}:=\int_E \Psi(u^1,u^2)\Psi(u^1,u^3)d\nu(u)\geq 0$ since it is the covariance of the infinite exchangeable sequence $(\Psi(U^1,U^k))_{k\geq 2}$, where $(U^k)_{k\geq 1}\sim\nu$. And if $i=j,$
$$\langle J^{N,i}\rangle_t=\frac1N\sum_{k=1, k \neq i }^N\int_0^t\int_E \Psi(u^i,u^k)^2f(X^{N,k}_s,\mu^N_s)d\nu(u)ds=\varsigma^2\int_0^t \int_\r f(x,\mu^N_s)\mu^N_s(dx)ds+ O (\frac{ t}{N}) ,$$
where $ \varsigma^2 = \int_E \Psi(u^1,u^2)^2 d\nu(u).$ 

As in Example~\ref{commonW}, there must be a common Brownian motion since the quadratic covariations between different particles are not zero. But, here $ \langle J^{N,i},J^{N,j}\rangle_t \neq  \langle J^{N,i}\rangle_t $ if $j \neq i .$ That is why there must be additional Brownian motions. Formally, the limit system in this example is
\begin{align}
d\bar X^i_t=&b(\bar X^{i}_t,\mu_t)dt+\sigma(\bar X^{i}_t,\mu_t)d\beta^i_t\label{eq:combine}\\
&+ {\xi } \sqrt{\int_\r f(x,\mu_t)\mu_t(dx)}dW_t+\sqrt{(\varsigma^2- {\xi^2})\int_\r f(x,\mu_t)\mu_t(dx)}dW^i_t,\nonumber
\end{align}
where we know that $\varsigma^2\geq\xi^2$ by Cauchy-Schwarz's inequality. As before, $W,W^i$ ($i\geq 1$) are independent standard one-dimensional Brownian motions, and $\mu=\mathcal{L}(\bar X^1|W)$ is random in this case. Note  that in the case where $\Psi (x,y,m,u^1,u^2)=\Psi(u_1)$,  we have  $\xi^2=\varsigma^2$, and in the case $\Psi (x,y,m,u^1,u^2)=\Psi(u_2),$ we have $\xi^2=0$,  hence this example covers both Example \eqref{commonW} and Example \eqref{iidW}. 
\end{ex}

Before defining the limit system in the general case, let us explain the main difficulty that arises. If we apply the same reasoning as in Examples~\ref{commonW},~\ref{iidW} and~\ref{combine} to the general model given in~\eqref{XNi}, we obtain for two 
 different particles 
\begin{align*}
&\langle J^{N,i},J^{N,j}\rangle_t=\\
&\frac1N\sum_{k=1, k \neq i, j }^N\int_0^t f(X^{N,k}_s,\mu^N_s)\int_E\Psi(X^{N,k}_s,X^{N,i}_s,\mu^N_s,u^k,u^i)\Psi(X^{N,k}_s,X^{N,j}_s,\mu^N_s,u^k,u^j)\nu(du)ds\\
&=\int_0^t\int_\r f(x,\mu^N_s)\int_E \Psi(x,X^{N,i}_s,\mu^N_s,u^1,u^2)\Psi(x,X^{N,j}_s,\mu^N_s,u^1,u^3)\nu(du)\mu^N_s(dx)ds + O (\frac{ t}{N}),
\end{align*}
under appropriate conditions on $\Psi,$ see Assumption \ref{hyppsi}  below.  And  for the quadratic variation of the jump term of a single particle  we get
$$\langle J^{N,i}\rangle_t=\int_0^t\int_\r  f(x,\mu^N_s) \int_E \Psi(x,X^{N,i},\mu^N_s,u^1,u^2)^2\nu(du)\mu^N_s(dx)ds,$$
still up to an error term of order $1/N.$ 

Contrarily to the situation of the previous examples, the quadratic covariations depend on the positions of the particles $i$ and $j$ and 
 can only be written as integrals of products where this integration involves, among others, the empirical measure of the process.  This is the reason why we need to use martingale measures and white noises instead of Brownian motions, as introduced in~\cite{walsh_introduction_1986}, confer also to~\cite{el_karoui_martingale_1990}. 

Let us briefly explain why using martingale measures is well adapted to our problem.  If $M$ is a martingale measure on $\r_+\times F$ (with $(F, {\cal F})$ some measurable space), having intensity $dt\cdot m_t(dy),$ then for all $ A, B \in {\cal F},$  $M_t(A):=M([0,t]\times A); t\geq 0$ is a square-integrable martingale and 
$$\ll\langle M_\cdot(A),M_\cdot(B)\rr\rangle_t=\int_0^t\int_F \un_{A\cap B}(y)m_s(dy)ds.$$

Having this remark in mind, it is natural to write the limit system in a similar way as~\eqref{eq:combine}, but replacing the Brownian motions by martingale measures. More precisely, under appropriate conditions on the coefficients, the limit system $ (\bar X^i)_{i \geq 1 }$ of \eqref{XNi} will be shown to be of the form
\begin{eqnarray}
d\bar X^i_t&=&b(\bar X^{i}_t,\mu_t)dt+\sigma(\bar X^{i}_t,\mu_t)d\beta^i_t\label{barXi}
+\int_\r\int_\r\sqrt{f(x,\mu_t)}\tilde\Psi(x,\bar X^i_t,\mu_t,v)dM(t,x,v) , \\
&&+\int_\r\sqrt{f(x,\mu_t)}\kappa(x,\bar X^i_t,\mu_t)dM^i(t,x), \;  i \geq 1,\nonumber\\
(\bar X^i_0)_{i\geq 1}&\sim& \nu_0^{\otimes\mathbb{N}^*}\nonumber .
\end{eqnarray}
In the above formula, 
\begin{align}
 \mu_t:=&\mathcal{L}(\bar X^i_t|W), \label{mut}\\
\tilde \Psi(x,y,m,v) :=& \int_\r \Psi(x,y,m,v,u^1)d\nu_1(u^1),\label{psitilde}\\
\kappa(x,y,m)^2:=&\int_E\Psi(x,y,m,u^1,u^2)^2d\nu(u)-\int_\r\tilde\Psi(x,y,m,v)^2d\nu_1(v)\nonumber\\
=&\int_E\Psi(x,y,m,u^1,u^2)^2d\nu(u)-\int_E\Psi(x,y,m,u^1,u^2)\Psi(x,y,m,u^1,u^3)d\nu(u) .\label{kappacarre}
\end{align}
Notice that the expression~\eqref{kappacarre} is positive by Cauchy-Schwarz's inequality.

In the above equations,  $M (dt, dx, dv )$ and $M^i (dt, dx) $ are orthogonal martingale measures on $ \r_+ \times \r \times \r $ ($\r_+ \times \r$ respectively) with respective intensities $dt\cdot\mu_t(dx)\cdot\nu_1 (dv) $ and $dt\cdot\mu_t(dx),$ defined as
\begin{equation}\label{eq:M}
M^i_t(A):=\int_0^t\un_A(F_s^{-1}(p))dW^i(s,p)\textrm{ and }M_t(A\times B):=\int_0^t\un_A(F_s^{-1}(p))\un_B(v)dW(s,p,v),
\end{equation}
with $W$ a white noise on $\r_+^2\times\r$ with intensity $dt\cdot dp\cdot d\nu_1(v),$  and $W^i$ ($i\geq 1$) independent white noises on $\r_+^2$, independent from $W,$ with intensity $dt\cdot dp .$ In the above formula, $F_s(x):=\pro{\bar X^i_s\leq x|W}$ is the conditional distribution function, conditionally on $W,$ and $F_s^{-1}$ is the generalized inverse of $F_s.$ As in~\eqref{XNi}, we assume that the Brownian motions, the white noises and the initial conditions are independent.

In the case where $\Psi(x,y,m,u_1,u_2)=\Psi (u_1,u_2),$ we see that $\kappa^2=\varsigma^2- {\xi^2}$ is a constant and that
$\tilde \Psi $ does only depend on $v$ such  that we can represent the two integrals with respect to the martingale measures in \eqref {barXi}  as two integrals against Brownian motions, recovering all previous examples.

Let us give some comments on the above system of equations. We have already argued that, in general,  $\mu_t$ is a random measure because of the scaling $N^{-1/2} .$ We shall prove that $\mu_t$ is actually the law of $\bar X^1$ conditionally on the common noise of the system. This common noise is the white noise $W$  underlying the martingale measure~$M.$
It is not immediately obvious that the definition of the martingale measures $ M$ and $M^i $ in \eqref{eq:M} and the limit system \eqref{barXi} are well-posed. In what follows, we shall give conditions ensuring that equation~\eqref{barXi} admits a unique strong solution. This is the content of our first main theorem, Theorem \ref{existencesolution}. To prove this theorem, we propose a  Picard iteration in which we construct  a sequence of martingale measures whose intensities depend on the conditional law of the instance of the process within the preceding step. One main ingredient of the proof is the well-known fact that the Wasserstein-$2-$distance of the laws of two real-valued random variables is given by the $L^2-$ distance of their inverse distribution functions - we apply this fact here to the conditional distribution functions. 

Using arguments that are inspired by \cite{erny_conditional_2020}, we then show in our second main theorem, Theorem \ref{convergencemuN}, that the finite particle system converges to the limit system, that is, $(X^{N,i})_{1\leq i\leq N}$ converges to $(\bar X^i)_{i\geq 1}$ in distribution in $D(\r_+,\r)^{\n^*}.$ This convergence is the consequence of the well-posedness of an associated martingale problem.  
 Contrarily to \cite{erny_conditional_2020} the finite system here depends on the empirical measure, and the conditions on the regularity of its coefficients are formulated in terms of the Wasserstein distance. To reconcile the convergence in distribution of the empirical measure with the Wasserstein-Lipschitz continuity of the coefficients gives an additional technical difficulty to the proof.

{\bf Organization of the paper.} In Section~\ref{secnota}, we state the assumptions and formulate the main results. Section \ref{sec:3} is devoted to the proof of Theorem~\ref{existencesolution}. The proofs of Theorems~\ref{convergenceXNi} and~\ref{convergencemuN} are gathered in Section \ref{sec:4}. Finally, in Section \ref{sec:multipop} we discuss extensions of our results to the frame of multi-populations where the particles are organized within clusters.

{\bf General notation.}
Throughout this paper we shall use the following notation. Given any measurable space $ (S, \mathcal S), $ $\mathcal{P} (S)$ denotes the set of all probability measures on $ (S, \mathcal S),$ endowed with the topology of weak convergence. For $p\in\n^*,$ $\mathcal{P}_p(\r)$ denotes the set of probability measures on $\r$ that have a finite moment of order~$p$.  
For two probability measures $\nu_1, \nu_2 \in \mathcal{P}_p (\r),$ the Wasserstein distance of order $p$ between $\nu_1$ and $\nu_2$  is defined as
$$
W_p(\nu_1,\nu_2)=\inf_{\pi\in\Pi(\nu_1,\nu_2)}\left( \int_S\int_S |x-y|)^p \pi(dx,dy) \right)^{1/p} ,
$$
where $\pi$ varies over the set $\Pi(\nu_1,\nu_2)$ of all probability measures on the product space $\r\times \r$ with marginals $\nu_1$ and $\nu_2$. Notice that  the Wasserstein distance of order $p$ between $\nu_1$ and $\nu_2$ can be rewritten as  the infimum of $E[| X - Y|^p]^{1/p}$ over all possible couplings $(X,Y)$ of the random elements $X$ and $Y$ distributed according to $\nu_1$ and $\nu_2$ respectively, i.e.
$$
W_p(\nu_1,\nu_2)=\inf\left\{\esp{|X - Y|^p}^{1/p}: \mathcal{L}(X)=\nu_1\ \mbox{and} \ \mathcal{L}(Y)=\nu_2  \right\}.
$$
Moreover, $D(\r_+,\r)$ (or just $D$ for short) denotes the space of c\`adl\`ag functions from $\r_+$ to $\r$, endowed with the Skorokhod metric, and $C$ and $K$ denote arbitrary positive constants whose values can change from line to line in an equation. We write $C_\theta$ and $K_\theta$ if the constants depend on some parameter $\theta.$
Finally, for any $n,p\in\n^*,$ we note $C_b^n(\r^p)$ (resp. $C_b^n(\r^p,\r_+)$) the set of real-valued functions $g$ (resp. non-negative functions $g$) defined on $\r^p$ which are $n$ times continuously differentiable such that $g^{(k)}$ is bounded for each $0\leq k\leq n .$ 

\section{Assumptions and main results.}\label{secnota}
\subsection{Assumptions}
We start imposing a hypothesis under which equation~\eqref{XNi} admits a unique strong solution and which grants a Lipschitz condition on the coefficients of the SDE.

\begin{assu}
\label{XNiwellposed}$ $
\begin{description}
\item {i)} For all $x,y\in\r,m,m'\in\mathcal{P}_1(\r),$
$$|b(x,m)-b(y,m')|+|\sigma(x,m)-\sigma(y,m')|\leq C(|x-y|+W_1(m,m')).$$
\item {ii)} $f$ is bounded and strictly positive, and $ \sqrt{f} $ is Lipschitz, that is, for all $x,y\in\r,m,m'\in\mathcal{P}_1(\r),$
$$|\sqrt{f}(x,m)-\sqrt{f}(y,m')|\leq C(|x-y|+W_1(m,m')).$$
\item {iii)}
For all $x,x',y,y',u,v\in\r,m,m'\in\mathcal{P}_1(\r),$
$$|\Psi(x,y,m,u,v)-\Psi(x',y',m',u,v)|\leq M(u,v)(|x-x'|+|y-y'|+W_1(m,m')),$$
where $M:\r^2\rightarrow\r_+$ satisfies $\int_E M(u^1,u^2)^2d\nu(u)<\infty.$
\item {iv)}
$$\underset{x,y,m}{\sup}\int_E |\Psi(x,y,m,u^1,u^2)|d\nu(u)<\infty.$$
\end{description}
\end{assu}
Notice that $ f$ bounded together with $ \sqrt{f} $ Lipschitz implies that $f$ is Lipschitz as well. As a consequence, relying on Theorem~2.1 of \cite{graham_mckean-vlasov_1992}, Assumption~\ref{XNiwellposed} implies that equation~\eqref{XNi} admits a unique strong solution.

In order to prove the well-posedness of the limit equation~\eqref{barXi}, we need additional assumptions. Recall that $ \kappa^2 $ has been introduced in \eqref{kappacarre} above. 
\begin{assu}\label{barXiwellposed}$ $
\begin{description}
\item {i)}
$$\underset{x,y,m}{\inf}\kappa(x,y,m)>0,$$
\item {ii)} $$\underset{x,y,m}{\sup}\int_E \Psi(x,y,m,u^1,u^2)^2d\nu(u)<\infty.$$
\end{description}
\end{assu}
\begin{rem}\label{kappa2lip}
Using the third point of Assumption~\ref{XNiwellposed} we can prove that $\kappa^2$ is Lipschitz continuous with Lipschitz constant proportional to $$\left (\int_{E}M^2(u,v)d\nu(u,v)\times \sup_{x,y,m}\int_E\Psi^2(x,y,m,u,v) d\nu(u,v)\right )^{1/2}.$$
Assumption~\ref{barXiwellposed}.$i)$ allows then to prove that $\kappa$ is Lipschitz continuous. Assumption ~\ref{barXiwellposed}.$ii)$ gives that $\|\kappa\| _{\infty}:=\sup_{x,y,m}\kappa(x,y,m)<\infty.$
\end{rem}

To prove the convergence of the particle system $(X^{N,i})_{1\leq i\leq N}$ to the limit system, we need further assumptions on the function $\Psi.$
\begin{assu}
\label{hyppsi}$ $
\begin{description}
\item {i)} For all $x,y\in\r,m\in\mathcal{P}_1(\r),$
$$\int_E \Psi(x,y,m,u^1,u^2)d\nu(u)=0,$$
\item {ii)} $$ \int_E \underset{x,y,m}{\sup}|\Psi(x,y,m,u^1,u^2)|^3d\nu(u)<\infty.$$
\item {iii)} $b$ and $ \sigma $ are bounded.
\end{description}
\end{assu}

\begin{rem}
We assume the functions $b$ and $\sigma$ to be bounded to simplify the proofs of Lemmas~\ref{apriorimckeandiffu} and~\ref{suitethm43}. However the results of these lemmas still hold true under the following weaker assumption: there exists $C>0$ such that, for all $x\in\r,m\in\mathcal{P(\r)},$
$$|b(x,m)| + |\sigma(x,m)|\leq C(1+|x|).$$
In other words, $b$ and $\sigma$ are bounded w.r.t. the measure variable and sublinear w.r.t. the space variable.
\end{rem}

Let us give an example of a function $\Psi$ that satisfies all our assumptions and where the random quantity $\Psi$ depends on the difference of the states of the jumping and the receiving particle as well as on the average state of the system as follows
$$\Psi(x,y,m,u,v)=uv\ll(\eps+\frac{\pi}{2}+\arctan(x- y+\int_\r zdm(z))\rr),$$
 with $\nu=\mathcal{R}^{\otimes\n^*}$ and $\mathcal{R}= \frac12 ( \delta_{-1} + \delta_{+1} ) $ the Rademacher distribution. In the formula above, the variables $u$ and $v$ can be seen as spins such that the receiving particle is excited if the orientation of its spin is the same as the spin of the sending particle, and inhibited otherwise. $\Psi$ satisfies the Lipschitz condition of Assumption~\ref{XNiwellposed} because $\arctan$ is Lipschitz continuous. The hypothesis on the moments of $\Psi$ are also satisfied since $\mathcal{N}(0,1)$ has finite third moments and is centered. Finally, the first point of Assumption~\ref{barXiwellposed} holds true, because
\begin{multline*}
\kappa(x,y,m)^2=\int_E\Psi(x,y,m,u^1,u^2)^2d\nu(u)-\int_E\Psi(x,y,m,u^1,u^2)\Psi(x,y,m,u^1,u^3)d\nu(u)\\
=\ll(\eps+\frac{\pi}{2}+\arctan(x- y+\!\! \int_\r \!\!\!zdm(z))\rr)^2
\ll(\int_E \!\! (u^1u^2)^2d\nu(u)-\int_E u^1u^2u^1u^3d\nu(u)\rr)\\
=\ll(\eps+\frac{\pi}{2}+\arctan(x-y+\int_\r zdm(z))\rr)^2.
\end{multline*}

\subsection{Main results}

Our first main result is the well-posedness of the limit equation.

\begin{thm}
\label{existencesolution}
Under Assumptions~\ref{XNiwellposed} and~\ref{barXiwellposed}, equation~\eqref{barXi} admits a unique strong solution $\bar X^i  $ that possesses finite second moments. This solution also has finite fourth moments.
\end{thm}

Our second main result states the convergence of $(X^{N,i})_{1\leq i\leq N}$ to $(\bar X^i)_{i\geq 1}$.

\begin{thm}
\label{convergenceXNi}
Under Assumptions~\ref{XNiwellposed},~\ref{barXiwellposed} and~\ref{hyppsi}, $(X^{N,i})_{1\leq i\leq N}$ converges to $(\bar X^i)_{i\geq 1}$ in distribution in $D(\r_+,\r)^{\n^*}.$
\end{thm}
In the above statement, we implicitly define $X^{N,i}:=0$ if $i>N.$

As the systems $(X^{N,i})_{1\leq i\leq N}$ ($N\in\n^*$) and $(\bar X^i)_{i\geq 1}$ are exchangeable, Theorem~\ref{convergenceXNi} is equivalent to
\begin{thm}
\label{convergencemuN}
Under Assumptions~\ref{XNiwellposed},~\ref{barXiwellposed} and~\ref{hyppsi}, the system $ (\bar X^i )_{ i \geq 1 } $ is exchangeable with directing measure $ \mu =\mathcal{L}(\bar X^1|W),$ where $ W$ is as in \eqref{eq:M}. Moreover, the sequence of empirical measures $$\mu^N:=N^{-1}\sum_{i=1}^N\delta_{X^{N,i}}$$ converges in law, as $\PP(D(\r_+,\r))$- valued random variables, to $\mu.$ 
\end{thm}

\begin{rem}
In Theorem~\ref{convergencemuN}, it is easy to prove that $\mathcal{L}(\bar X^1|W)$ is the directing measure of the system $(\bar X^i)_{i\geq 1}$. Indeed, it is sufficient to notice that conditionally on $W,$ the variables $\bar X^i$ ($i\geq 1$) are i.i.d. and to apply Lemma~(2.12) of \cite{aldous_exchangeability_1983}.
\end{rem}

The proof of Theorem~\ref{convergencemuN} is similar to the proof  of Theorem~1.7 of \cite{erny_conditional_2020}. It consists in showing that $(\mu^N)_N$ is tight on $\PP (D(\r_+,\r)),$ and that each converging subsequence converges to the same limit, using a convenient martingale problem. For this reason, in what follows we just give the proofs that substantially change compared to our previous paper. For the other proofs, we just give the main ideas and cite precisely the corresponding statement of \cite{erny_conditional_2020} that allows to conclude.

\section{Proof of Theorem~\ref{existencesolution}}\label{sec:3} 

The following Lemma shows that the definition \eqref{eq:M} indeed defines the right martingale measures.

\begin{lem}\label{lem:changevar}
Let $\nu_1$ be a probability measure on~$\r$. Moreover, let $(\Omega , {\cal A}, P )$ be a probability space, $(\mathcal{G}_t)_t$ be a filtration on it, $(\mathcal{F}_t)_t$ be a sub-filtration of $ (\mathcal{G}_t)_t$ and $W$ (resp. $W^1$) be a $(\mathcal{G}_t)_t-$white noise on $\r_+\times[0,1]\times\r$ (resp. $\r_+\times[0,1]$) of intensity $dt\cdot dp\cdot\nu_1(dv)$ (resp. $dt\cdot dp$). Let $X$ be a continuous $\r-$valued process which is $(\mathcal{G}_t)_t$-adapted and $F_s(x):=\mathbb{P}(X_s\leq x|\mathcal{F}_s).$  Moreover, we suppose that for all $s>0,$ $\mathbb{P}(X_s\leq x|\mathcal{F}_s)=\mathbb{P}(X_s\leq x|\mathcal{F}_{\infty}),$ where $\F_{\infty}=\sigma\{\F_t,\ t\geq 0\}.$

Define for any $ A, B  \in \mathcal B ( \r)  , $ 
$$M^1_t(A):=\int_0^t\int_0^1 \un_A((F_{s})^{-1}(p))dW^1(s,p) , \quad  M_t(A\times B):=\int_0^t\int_0^1\int_\r \un_A((F_{s})^{-1}(p))\un_B(v)dW(s,p,v)$$
Then $M^1 $ and $M$ are martingale measures with respective intensities $dt\cdot \mu_t(dx)$ and $dt\cdot \mu_t(dx)\cdot \nu_1(dv)$, where $\mu_t:=\mathcal{L}(X_t|\mathcal{F}_t).$ 
\end{lem}

\begin{proof}
We only show the result for the martingale measure $M^1$. The main part of the proof consists in showing that the process $(\omega,s,p)\in\Omega\times\r_+\times[0,1]\mapsto (F_s)^{-1}(p)$ is $\mathcal{P}\otimes\mathcal{B}([0,1])-$measurable, with $\mathcal{P}$ the predictable sigma field related to the filtration $(\mathcal{G}_t)_t$.

To begin with, let us prove that $(\omega,s,x)\mapsto F_s(x)$ is $\mathcal{P}\otimes\mathcal{B}(\r)-$measurable. We write
$$F_s(x)=P(X_s\leq x|\mathcal{F}_s)=\esp{\phi(x,X_s)|\mathcal{F}_s},$$
where $\phi(x,y):=\uno{x\leq y}.$ As $\phi$ is product measurable and bounded, it is the limit of functions of the form
$$\sum_{k=1}^nc_k\phi_k(x)\psi_k(y),$$
where the functions $\phi_k,\psi_k$ ($1\leq k\leq n$) are Borel measurable and bounded. This limit can be taken to be increasing such that, by monotone convergence,
$$\esp{\phi(x,X_s)|\mathcal{F}_s}=\underset{n}{\lim}\sum_{k=1}^n c_k \phi_k(x)\esp{\psi_k(X_s)|\mathcal{F}_s}.$$
Then, as $\psi_k$ is a bounded and Borel function, it can be approximated by an increasing sequence of bounded and continuous functions $\psi_{k,m}.$ Then, as for every $n,m,$ the process
 $$(\omega,s,x)\mapsto \sum_{k=1}^nc_k\phi_k(x)\esp{\psi_{k,m}(X_s)|\mathcal{F}_s}=\sum_{k=1}^nc_k\phi_k(x)\esp{\psi_{k,m}(X_s)|\mathcal{F}_{\infty}}$$
is continuous in~$s$ and $(\mathcal{F}_s)_s-,$ whence $(\mathcal{G}_s)_s-$adapted, it is $\mathcal P\otimes\mathcal{B}(\r)-$measurable.

Let $x\in\r$ be fixed. It is sufficient to show that $\{(\omega,s,p):(F_s)^{-1}(p)\geq x\}$ is measurable. Let us write
$$\{(\omega,s,p):(F_s)^{-1}(p)\geq x\}=\{(\omega,s,p):F_s(x)\leq p\}=\{(\omega,s,p):\phi(F_s(x),p)>0\},$$
where $\phi(x,p):=\uno{x\leq p}$ is product measurable.

Then, the measurability of $\{(\omega,s,p):(F_s)^{-1}(p)\geq x\}$ is a consequence of that of $(\omega,s,p)\mapsto \phi(F_s(x),p)$ w.r.t. $\mathcal{P}\otimes\mathcal{B}([0,1]).$

As a consequence, the process $(\omega,s,p)\in\Omega\times\r_+\times[0,1]\mapsto (F_s)^{-1}(p)$ is $\mathcal{P}\otimes\mathcal{B}([0,1])-$measurable. The rest of the proof consists in writing
$$\langle M^1_\cdot (A)\rangle_t=\int_0^t\int_0^1\un_A((F_s)^{-1}(p))dpds=\int_0^t\mu_s(A)ds.$$
The last inequality above is a classical property of the generalized inverse of the distribution function (see e.g. Fact 1 in Section 8 of \cite{major_invariance_1978}).
\end{proof}

\subsection{Construction of a strong solution of \eqref{barXi} - proof of Theorem \ref{existencesolution}}
We construct a strong solution of \eqref{barXi} using a Picard iteration.
 Let $ \beta^i, i\in\n^*,$ be independent one-dimensional Brownian motions. Let $W,W^i, \ i\in\n^*,$ be independent white noises on respectively $\r_+\times [0,1]\times \r$ and $\r_+\times [0,1]$ with respective intensities $dt\cdot dp\cdot \nu_1(dv)$ and $dt\cdot dp,$ independent of the $ \beta^i.$ We suppose that all these processes are defined on the same probability space $(\Omega,\F,P)$ carrying also i.i.d. random variables $X^i_0,\, i\in \n^*,$ which are independent of the $ \beta^i, W, W^i .$ 
Define for all $t\geq 0,$ $$\G_t:=\sigma\{\, \beta^{i}_v; \,  W(]u,v]\times A\times B);\,  W^i(]u,v]\times A);\,  0<u<v\leq t;\,  A\in\B(\r);\,  B\in\B([0,1]);\, i\in\n^*\, \};$$
$$\W_t:=\sigma\{\, W(]u,v]\times A\times B);  u<v\leq t;\  A\in\B(\r), B\in\B([0,1])\, \};$$
$$\W:=\sigma\{W_t;\  t\geq 0\}.$$

{\it Step~1.}  Fix an $i\in\n^*,$
and  introduce 
\begin{align*}
X^{i,[0]}_t:=&X^i_0,\\
\mu^{[0]}_t:=&\mathcal{L}(X^i_0), \; 
F^{[0]}_t(x):=\pro{X_0\leq x|\W},\\
M^{[0]}([0,t]\times A\times B):=&\int_0^t\int_0^1 \int_{\r}\un_A((F^{[0]}_s)^{-1}(p))\un_B(v)dW(s,p,v)\\
M^{i,[0]}([0,t]\times A):=&\int_0^t\int_0^1 \un_A((F^{[0]}_s)^{-1}(p))dW^i(s,p).
\end{align*}

Assuming everything is defined at order $n\in\mathbb{N},$ we introduce
\begin{align}
X^{i,[n+1]}_t:=&\int_0^t b(X^{i,[n]}_s,\mu^{[n]}_s)ds+\int_0^t\sigma(X^{i,[n]}_s,\mu^{[n]}_s)d\beta^i_s\label{Xin}\\
&+\int_0^t\int_\r\int_\r\sqrt{f(x,\mu^{[n]}_s)}\tilde\Psi(x,X^{i,[n]}_s,\mu^{[n]}_s,v)dM^{[n]}(s,x,v)\nonumber\\
&+\int_0^t\int_\r\sqrt{f(x,\mu^{[n]}_s)}\kappa(x,X^{i,[n]}_s,\mu^{[n]}_s)dM^{i,[n]}(s,x),\nonumber\\
\mu^{[n+1]}_s:=&\mathcal{L}(X_s^{i,[n+1]}|\W), \, 
F_s^{[n+1]}(x):= \mathbb{P}(X_s^{i,[n+1]}\leq x|\W),\nonumber\\
M^{[n+1]}([0,t]\times A\times B):=&\int_0^t\int_0^1\int_\r\un_A((F_s^{[n+1]})^{-1}(p))\un_B(v)dW(s,p,v),\nonumber\\
M^{i,[n+1]}([0,t]\times A):=&\int_0^t\int_0^1\un_A((F_s^{[n+1]})^{-1}(p))dW^i(s,p).\nonumber
\end{align}
Note that $\forall t>0,$ $$\W=\sigma(\W_t;\,  \W_{]t;\infty[})$$
where $$\W_{]t;\infty[}:=\sigma\{]W(]u,v]\times A);\ , t<u<v;\,   A\in\B(\r)\}.$$ Remember also that white noises are processes with independent increments, more precisely, for all $A,A'$ in $\B([0,1])$, for all $B,B'$ in $\B(\r)$,
$W(]u,v]\times A\times B)$ and $W(]u',v']\times A'\times B')$ are independent if $]u;v]\cap ]u';v']=\emptyset.$

Using this last remark, we see  that by construction $X^{i,[n+1]}_t$ is independent from $\W_{]t;\infty[}$, and as a consequence, $\mathbb{P}(X_s^{i,[n+1]}\leq x|\W)=\mathbb{P}(X_s^{i,[n+1]}\leq x|\W_s)$. Taking $\F_t=\W_t$ and $\F_{\infty}=\W$ we see that all assumptions of Lemma \eqref{lem:changevar} are satisfied. Hence,  for each $n\in\n;\, i\in\n^* $ the martingale measures $M^{[n]}$ and $M^{i,[n]}$ are well defined and have respectively the intensities $dt\cdot \mu^{[n]}_t(dx)\cdot \nu_1(dv)$ and $dt\cdot \mu^{[n]}_t(dx)$, where $\mu^{[n]}_t:=\mathcal{L}(X^{i, [n]}_t|\W)=\mathcal{L}(X^{i, [n]}_t|\W_t)$    .  

In what follows, we shall consider $u^{[n]}_t:=\esp{\ll(X^{i,[n+1]}_t-X^{i,[n]}_t\rr)^2}.$ Let us introduce
$$h(x,y,m,v):=\sqrt{f(x,m)}\tilde\Psi(x,y,m,v) \mbox{ and } g(x,y,m):=\sqrt{f(x,m)}\kappa(x,y,m).$$
Note that the assumptions of the theorem  guarantee that $h$ and $g$ are Lipschitz continuous. Indeed, using Assumption~  \ref{XNiwellposed} $(ii)$ and $(iii)$, 
for all $x,y,x',y',v\in\r,m,m'\in\mathcal{P}_1(\r),$
$$|h(x,y,m,v)-h(x',y',m',v)|\leq C(v)(|x-x'|+|y-y'|+W_1(m,m')),$$
where 
$$C(v):=\int_E M(u,v)\nu_1(du)+\int_{E} |\Psi(x,y,m,v,u)|\nu_1(du).$$

Using Jensen's inequality together with Assumptions~ \ref {XNiwellposed} $(iii)$  and \ref {barXiwellposed} $(ii)$ we see that $C$ 
satisfies $\int_\r C(v)^2dv<\infty.$ Moreover, using Assumption~ \ref {XNiwellposed} $(ii)$  together with Remark~ \ref{kappa2lip},   for all $x,x',y,y'\in\r,m,m'\in\mathcal{P}_1(\r),$
$$|g(x,y,m)-g(x',y',m')|\leq K(|x-x'|+|y-y'|+W_1(m,m')),$$
where $K\leq C(\|\kappa\|_{\infty}+\sqrt{\|f\|_{\infty}}).$

{\it Step~2.} We now prove that our Picard scheme converges. Classical arguments imply the existence of a constant $C > 0 $ such that 

\begin{multline}\label{calcul2}
\frac{1}{C} \ll(X^{i,[n+1]}_t-X^{i,[n]}_t\rr)^2\leq \\
\ll(\int_0^t(b(X^{i,[n]}_s,\mu^{[n]}_s)-b(X^{i,[n-1]}_s,\mu^{[n-1]}_s))ds\rr)^2+\ll(\int_0^t(\sigma(X^{i,[n]}_s,\mu^{[n]}_s)-\sigma(X^{i,[n-1]}_s,\mu^{[n-1]}_s))d\beta^i_s\rr)^2\\
+\ll(\int_0^t\int_\r\int_\r h(x,X^{i,[n]}_s,\mu^{[n]}_s,v)dM^{[n]}(s,x,v)-\int_0^t\int_\r\int_\r h(x,X^{i,[n-1]}_s,\mu^{[n-1]}_s,v) dM^{[n-1]}(s,x,v)\rr)^2\\
+\ll(\int_0^t\int_\r g(x,X^{i,[n]}_s,\mu^{[n]}_s)dM^{i,[n]}(s,x)-\int_0^t\int_\r g(x,X^{i,[n-1]}_s,\mu^{[n-1]}_s) dM^{i,[n-1]}(s,x)\rr)^2\\
\leq t\int_0^t\ll(b(X^{i,[n]}_s,\mu^{[n]}_s)-b(X^{i,[n-1]}_s,\mu^{[n-1]}_s)\rr)^2ds+\ll(\int_0^t(\sigma(X^{i,[n]}_s,\mu^{[n]}_s)-\sigma(X^{i,[n-1]}_s,\mu^{[n-1]}_s))d\beta^i_s\rr)^2\\
+\ll(\int_0^t\int_0^1\int_\r \ll[h((F^{[n]}_s)^{-1}(p),\bar X^{i,[n]}_s,\mu^{[n]}_s,v) - h((F^{[n-1]}_s)^{-1}(p),\bar X^{i,[n-1]}_s,\mu^{[n-1]}_s,v) \rr]dW(s,p,v)\rr)^2\\
+ \ll(\int_0^t\int_0^1 \ll[g((F^{[n]}_s)^{-1}(p),\bar X^{i,[n]}_s,\mu^{[n]}_s) - g((F^{[n-1]}_s)^{-1}(p),\bar X^{i,[n-1]}_s,\mu^{[n-1]}_s) \rr]dW^i(s,p)\rr)^2 .
\end{multline}
Using Burkholder-Davis-Gundy's inequality to control the expectation of the stochastic integrals above, and using the fact that for all $\mu,\nu\in\mathcal{P}_2(\r),$
$$W_1(\mu,\nu)\leq W_2(\mu,\nu),$$
we have that
\begin{multline}\label{calcul1}
u^{[n]}_t\leq C(1+t)\int_0^t\esp{(X^{i,[n]}_s-X^{i,[n-1]}_s)^2}ds+C(1+t)\int_0^t\esp{W_2(\mu^{[n]}_s,\mu^{[n-1]}_s)^2}ds\\
\\+C\int_0^t\esp{\int_0^1((F_s^{[n]})^{-1}(p)-(F_s^{[n-1]})^{-1}(p))^2dp}ds.
\end{multline}
A classical result (see e.g. Theorem~8.1 of \cite{major_invariance_1978}) states that, if $F,G$ are two distribution functions with associated probability measure $ \mu $ and $ \nu, $ respectively, then
$$\int_0^1(F^{-1}(p)-G^{-1}(p))^2dp=\underset{X\sim \mu,Y\sim \nu }{\inf}\esp{(X-Y)^2}=W_2(\mu, \nu ),$$
where the infimum is taken over all possible couplings $ ( X, Y) $ of $ \mu$ and $ \nu.$

This implies that
$$\int_0^1((F_s^{[n]})^{-1}(p)-(F_s^{[n-1]})^{-1}(p))^2dp=W_2(\mu^{[n]}_s,\mu^{[n-1]}_s)^2.$$
Since $X^{i,[n]}$ and $X^{i,[n-1]}$, conditionally on $\W,$ are respectively realizations of $F^{i,[n]}$ and $F^{i,[n-1]}$,  we have that, for every $s\geq 0,$ almost surely,
$$W_2(\mu^{[n]}_s,\mu^{[n-1]}_s)^2\leq \esp{(X^{i,[n]}_s-X^{i,[n-1]}_s)^2|\W}.$$
Integrating with respect to $W$  implies that 
$$ \esp{\int_0^1((F_s^{[n]})^{-1}(p)-(F_s^{[n-1]})^{-1}(p))^2dp} \le  \esp{(X^{i,[n]}_s-X^{i,[n-1]}_s)^2}.$$
Consequently, we have shown that there exists some constant $C>0$ such that, for all $t\geq 0$,
\begin{equation}
\label{ununmoinsun}
u^{[n]}_t\leq C(1+t)\int_0^tu^{[n-1]}_sds.
\end{equation}
Classical computations then give
$$u^{[n]}_t\leq C^n(1+t)^n\frac{t^n}{n!}.$$
Now, introducing $v^{[n]}_t:=2^{n}u^{[n]}_t,$ we have that
$$\sum_{n\geq 0}v^{[n]}_t<\infty.$$
Hence, using that for all  $x\in\r,\eps>0,$ $ |x|\leq\max(\eps,x^2/\eps)\leq \eps+ x^2/\eps,$ and applying this with $\eps=1/2^n$ and $x=X^{i,[n+1]}_t-X^{i,[n]}_t,$ we have
$$\sum_{n\geq 0}\esp{|X^{i,[n+1]}_t-X^{i,[n]}_t|}\leq \sum_{n\geq 0}\frac{1}{2^n}+\sum_{n\geq 0}v^{[n]}_t<\infty.$$
As a consequence, we can define, almost surely,
$$\bar X^i_t:=X^i_0+\sum_{n\geq 0}(X^{i,[n+1]}_t-X^{i,[n]}_t)<+\infty,$$
 and we know that $\esp{|\bar X^i_t - X^{i,[n]}_t|}$ vanishes as $n$ goes to infinity, and that $X^{i,[n]}_t$ converges almost surely to $\bar X^i_t$. 

{\it Step~3.} Let us prove that $\bar X^i$ has finite fourth moments. Let $w^{[n]}_t:=\esp{(X^{i,[n]}_t)^4}.$

By equation~\eqref{Xin}, we have
\begin{align}
\frac1C\ll(X^{i,[n]}_t\rr)^4\leq&\ll(\int_0^tb(X^{i,[n-1]}_s,\mu^{[n-1]}_s)ds\rr)^4+\ll(\int_0^t\sigma(X^{i,[n-1]}_s,\mu^{[n-1]}_s)d\beta^i_s\rr)^4\nonumber\\
&+\ll(\int_0^t\int_\r\int_\r h(x,X^{i,[n-1]}_s,\mu^{[n-1]}_s,v) dM^{[n-1]}(s,x,v)\rr)^4\nonumber\\
&+\ll(\int_0^t\int_\r g(x,X^{i,[n-1]}_s,\mu^{[n-1]}_s) dM^{i,[n-1]}(s,x)\rr)^4\nonumber\\
&=:A_1+A_2+A_3+A_4.\label{controlwnt}
\end{align}

First of all, let us note that our Lipschitz assumptions allow to consider the following control: for any $x\in\r,m\in\mathcal{P}_1(\r),$
$$|b(x,m)| \leq |b(x,m) - b(0,\delta_0)| + |b(0,\delta_0)| \leq C(1+|x|+W_1(\mu,\delta_0))= C\ll(1+|x| + \int_\r |y|dm(y)\rr),$$
and similar controls for the functions $\sigma,h,g.$ Using this control and Jensen's inequality, we have
\begin{align*}
\esp{A_1} \leq & t^3\int_0^t \esp{b(X^{i,[n-1]}_s,\mu^{[n-1]}_s)^4}ds\\
\leq& Ct^3\int_0^t \ll(1+w^{[n-1]}_s + \esp{\ll(\int_\r |y| \mu^{[n-1]}_s(dy)\rr)^4}\rr)ds\\
\leq& C t^3\int_0^t \ll(1+w^{[n-1]}_s\rr)ds.
\end{align*}

We can obtain a similar control for the expressions $A_2,A_3$ and $A_4$ using Burkholder-Davis-Gundy's inequality noticing that the stochastic integrals involved are local martingales. We just give the details for $A_2.$

\begin{align*}
\esp{A_2}\leq& \esp{\ll(\int_0^t\sigma(X^{i,[n-1]}_s,\mu^{[n-1]}_s)^2ds\rr)^2}\\
\leq& \esp{t\int_0^t\sigma(X^{i,[n-1]}_s,\mu^{[n-1]}_s)^4ds}\leq C t\int_0^t\ll(1+ w^{[n-1]}_s\rr)ds,
\end{align*} 
where the last inequality can be obtained with the same reasoning as the one used to control $\esp{A_1}.$ With the same reasoning, we have the following controls for $A_3$ and $A_4:$
$$\esp{A_3} + \esp{A_4} \leq C t\int_0^t\ll(1+ w^{[n-1]}_s\rr)ds.$$
Using the previous control in the inequality~\eqref{controlwnt}, we have that for all~$t\geq 0,$
$$w^{[n+1]}_t\leq C(1+t^3)+C(1+t^3)\int_0^t w^{[n]}_sds,$$
whence
$$w^{[n]}_t\leq \sum_{k=1}^n\frac{t^{k-1}}{(k-1)!}C^k(1+t^3)^k\leq C(1+t^3)e^{Ct(1+t^3)}.$$
Consequently
\begin{equation}
\label{controlXNmckeandiffu}
\underset{n\in\n}{\sup}~~\underset{0\leq s\leq t}{\sup}\esp{(X^{i,[n]}_s)^4}<\infty,
\end{equation}
for some constant $C>0.$ Then Fatou's lemma implies the result: for all $t\geq 0,$
\begin{equation}
\label{controlbarXmckeandiffu}
\underset{0\leq s\leq t}{\sup}\esp{(\bar X^i_s)^4}<\infty.
\end{equation}

{\it Step~4.} Finally, we conclude the proof showing that $\bar X^i$ is solution to the limit equation. Roughly speaking, this step consists in letting $n$ tend to infinity in~\eqref{Xin}.

We want to prove that, for all $t\geq 0,$
\begin{equation}
\label{equationlimite}
\bar X^i_t = G_t^i(\bar X^i,\mu),
\end{equation}
where
\begin{align*}
G_t^i(\bar X^i,\mu):=& \int_0^t b(\bar X^i_s,\mu_s)ds + \int_0^t\sigma(\bar X^i_s,\mu_s)d\beta^i_s\\
&+ \int_0^t\int_0^1\int_\r h(F_s^{-1}(p),\bar X^i_s,\mu_s,v)dW(s,p,v)\\
&+\int_0^t\int_0^1 g(F_s^{-1}(p),\bar X^i_s,\mu_s)dW^i(s,p),
\end{align*}
where the functions $h$ and $g$ have been introduced in {\it Step~1}, $\mu_t := \mathcal{L}(\bar X^i_t|\W)$, and $F_t^{-1}$ is the generalized inverse of
$$F_t(x) := \proc{\bar X^i_t\leq x}{\W}.$$
Let us note that $G^i_t$ has to be understood as a notation, we do not use its functional properties.

By construction, we have
\begin{equation}
\label{equationavantlimite}
X^{i,[n+1]}_t = G_t^i(X^{i,[n]},\mu^{[n]}).
\end{equation}

We have proved in {\it Step~2} that $X^{[n+1]}_t$ converges to $\bar X^i_t$ in $L^1.$ In other words, the LHS of~\eqref{equationavantlimite} converges to the LHS of~\eqref{equationlimite} in $L^1$. Now, it is sufficient to prove that the RHS converges in $L^2.$ This will prove that the equation~\eqref{equationlimite} holds true.

With the same computations as the ones used to obtain~\eqref{ununmoinsun} (recalling that this inequality relies on~\eqref{calcul2} and~\eqref{calcul1}), we have
$$\esp{\ll(G^i_t(\bar X^i,\mu) - G^i_t(X^{i,[n]},\mu^{[n]})\rr)^2}\leq C(1+t)\int_0^t \esp{\ll(\bar X^i_s - X^{i,[n]}_s\rr)^2}ds.$$

This proves that $G^i_t(X^{[n],i},\mu^{[n]})$ converges to $G^i_t(\bar X^i,\mu)$ in $L^2$ by dominated convergence: indeed, we know that for all $s\leq t,$ $X^{[n],i}_s$ converges to $\bar X^i_s$ almost surely thanks to {\it Step~2.}, and~\eqref{controlXNmckeandiffu} and~\eqref{controlbarXmckeandiffu} give the uniform integrability.

\subsection{Trajectorial uniqueness}
We continue the proof of Theorem  \ref{existencesolution} by proving the uniqueness of the solution. For that sake, 
let $\hat X^i$ and $\check X^i$ be two strong solutions defined with respect to the same initial condition $X^i_0$ and the same white noises $W$ and $W^i.$ Let
$$u_t:=\esp{(\hat X^i_t-\check X^i_t)^2}.$$

According to the computation of the previous subsection, there exists a constant $C>0$ such that for all $t\geq 0,$
$$u_t\leq C(1+t)\int_0^t u_sds.$$
Then Gr\"onwall's lemma implies that $u_t=0$ for all $t\geq 0,$ implying the uniqueness. 

\section{Proof of Theorems~\ref{convergenceXNi} and~\ref{convergencemuN}}\label{sec:4}

The proof of Theorem~\ref{convergencemuN} follows the same steps as the proof of Theorem~1.7 of \cite{erny_conditional_2020}: in a first time, we prove that the sequence $(\mu^N)_N$ is tight on $\mathcal{P} (D(\r_+,\r)),$ and then we prove that each converging subsequence of $(\mu^N)_N$ converges to the same limit, by proving that the limits of such subsequences are solutions to some martingale problem which is well-posed.

\subsection{Tightness of \texorpdfstring{$(\mu^N)_N$}{(muN)N}}
\label{tightnessmesuremartingale}

The proof that the sequence $(\mu^N)_N$ is tight on $\mathcal{P} (D(\r_+,\r))$, is almost the same as that of Proposition~2.1 in \cite{erny_conditional_2020}.  As the systems $(X^{N,i})_{1\leq i\leq N}$ ($N\in\n^*$) are exchangeable, it is equivalent to the tightness of the sequence $(X^{N,1})_N$ on $D(\r_+,\r)$ (see Proposition~2.2-(ii) of \cite{sznitman_topics_1989}).

The tightness of $(X^{N,1})_N$ is straightforward using Aldous' criterion (see Theorem~4.5 of \cite{jacod_limit_2003}), observing that, under our conditions, $ \sup_N \esp{ \sup_{s \le t } |X^{N, 1 }_s| } < \infty$ (see Lemma~\ref{apriorimckeandiffu}). 

\subsection{Martingale problem}\label{section:mp}
To identify the structure of any possible limit of the sequence $(\mu^N)_N$, we introduce a convenient martingale problem. We have already used such a kind of martingale problem in a similar context in Section~2.2 of \cite{erny_conditional_2020}. Since our limit system is necessarily an infinite exchangeable system, the martingale problem is constructed such that it reflects the correlations between the particles. It is therefore stated in terms of couples of particles.

Consider a probability measure  $Q\in\mathcal{P}(\mathcal{P} (D(\r_+,\r))).$ In what follows the role of $ Q$ will be to be the law of  any possible limit $ \mu $ of $\mu^N.$ Our martingale problem is stated on  the canonical space 
$$\Omega=\mathcal{P} (D(\r_+,\r))\times D(\r_+,\r)^2.$$
We endow $\Omega$  with the product of the associated Borel sigma-fields  and with the probability measure defined for all $A\in \B(\mathcal{P} (D(\r_+,\r))),\; B\in\B( D(\r_+,\r)^2)$ by
\begin{equation}\label{pq}
P_Q(A\times B):=\int_{\mathcal{P}_1(D(\r_+,\r))}\un_A(m) m\otimes m(B)Q(dm).
\end{equation}

We write an atomic event $\omega\in\Omega$ as $\omega=(m,y)$ with $y=(y_t)_{t\geq 0}=(y^1_t,y^2_t)_{t\geq 0}.$ 
We write $\mu$ and $Y=(Y^1,Y^2)$ for the random variables $\mu(\omega)=m$ and $Y(\omega)=y.$ 
The definition \eqref{pq} implies that $Q$ is the distribution of $\mu$, and that conditionally on $\mu,$  $Y^1$ and $Y^2$ are i.i.d. with distribution $\mu.$ More precisely, $\mu\otimes\mu$ is a regular conditional distribution of $(Y^1,Y^2)$ given $\mu.$

For $t\geq 0$ we write $m_t$ for the $t$-th marginal of $m:$
$m_t(C)=m(Y\in D^2(\r_+,\r);\ y_t\in C);\quad C\in \B(\r).$ We denote $\mu_t$ the r.v. $\mu_t(\omega)=m_t$
and consider the filtration $(\G_t)_{t\geq 0}$ given by
 $$\mathcal{G}_t=\sigma(Y_s,s\leq t)\vee\sigma(\mu_s (A);\; A\in\B(\r),\; s\leq t).$$  

For all $g\in C^2_b(\r^2),$
define \begin{align}\label{L}
Lg(y,m,x,v):=&b(y^1,m)\partial_{y^1}g(y)+b(y^2,m)\partial_{y^2}g(y)\\ \nonumber
&+\frac12\sigma(y^1,m)^2\partial^2_{y^1}g(y)+\frac12\sigma(y^2,m)^2\partial^2_{y^2}g(y)\\\nonumber
&+\frac12 f(x,m)\kappa(x,y^1,m)^2\partial^2_{y^1}g(y)+\frac12 f(x,m)\kappa(x,y^2,m)^2\partial^2_{y^2}g(y)\\\nonumber
&+\frac12 f(x,m)\sum_{i,j=1}^2\tilde\Psi(x,y^i,m,v)\tilde\Psi(x,y^j,m,v)\partial^2_{y^iy^j}g(y),
\end{align}
and put for all $t\geq 0$
\begin{equation}\label {M}
M_t^g=M_t^{g}(\mu,Y):=g(Y_t)-g(Y_0)-\int_0^t\int_\r\int_\r Lg(Y_s,\mu_s,x,v)\nu_1(dv)\mu_s(dx)ds.
\end{equation}

\begin{defi}\label{def:martpro}
$Q$ is solution to the martingale problem $(\mathcal{M})$ if
\begin{itemize}
\item $Q-$almost surely, $\mu_0=\nu_0,$
\item for all $g\in C^2_b(\r^2),$ $(M_t^g)_{t\geq 0}$ is a $(P_Q,(\mathcal{G}_t)_t)-$martingale.
\end{itemize}
\end{defi}

Let us state a first result that allows us to partially recover our limit equation~\eqref{barXi} from the martingale problem~$(\mathcal{M})$. It is the equivalent of Lemma~2.4 of \cite{erny_conditional_2020} in the framework of white noises.

\begin{lem}
\label{representationmartingale}
Grant Assumptions~\ref{XNiwellposed},~\ref{barXiwellposed} and~\ref{hyppsi}. Let $Q$ be a solution of $(\mathcal{M})$.Using the notation above, there exist on an extension of $(\Omega, \F, (\G_t)_{t\geq 0}, P_Q)$ two Brownian motions $\beta^1,\beta^2$ and three white noises $W^1,W^2,W$; defined on
$(\r_+\times [0,1],\;  \B(\r_+\times [0,1]),\;  dt\otimes dp)$
for $W^{i},\; i=1,2;$ and on $(\r_+\times [0,1]\times\r,\; \B(\r_+\times [0,1])\times\r),\; dt\otimes dp\otimes \nu_1(dv))$ 
for $W$,  
such that $\beta^1,\beta^2,W^1,W^2,W$ are all independent and such that $(Y_t)$ admits the representation
\begin{align*}
dY^1_t=&b(Y^1_t,\mu_t)dt+\sigma(Y^1_t,\mu_t)d\beta^1_t\\
&+\int_0^1\int_\r\sqrt{f(F_s^{-1}(p),\mu_t)}\tilde\Psi(F_s^{-1}(p),Y^1_t,\mu_t,v)dW(t,p,v)\\
&+\int_0^1\sqrt{f(F_s^{-1}(p),\mu_t)}\kappa(F_s^{-1}(p),Y^1_t,\mu_t)dW^1(t,p),\\
dY^2_t=&b(Y^2_t,\mu_t)dt+\sigma(Y^2_t,\mu_t)d\beta^2_t\\
&+\int_0^1\int_\r\sqrt{f(F_s^{-1}(p),\mu_t)}\tilde\Psi(F_s^{-1}(p),Y^2_t,\mu_t,v)dW(t,p,v)\\
&+\int_\r\sqrt{f(F_s^{-1}(p),\mu_t)}\kappa(F_s^{-1}(p),Y^2_t,\mu_t)dW^2(t,p),
\end{align*}
where $F_s$ is the distribution function related to $\mu_s$, and $F_s^{-1}$ its generalized inverse.
\end{lem}

\begin{proof}
Theorem~II.2.42 of \cite{jacod_limit_2003} implies that $Y$ is a continuous semimartingale with characteristics $(B,C)$ given by
\begin{align*}
B^i_t=&\int_0^t b(Y^i_s,\mu_s)ds,~~1\leq i\leq 2,\\
C^{i,i}_t=&\int_0^t\sigma(Y^i_s,\mu_s)^2ds+\int_0^t\int_\r f(x,\mu_s)\kappa(x,Y^i_s,\mu_s)^2\mu_s(dx)ds\\
&+\int_0^t\int_\r\int_\r f(x,\mu_s)\tilde\Psi(x,Y^i_s,\mu_s,v)^2\nu_1(dv)\mu_s(dx)ds,~~1\leq i\leq 2,\\
C^{1,2}_t=&\int_0^t\int_\r\int_\r f(x,\mu_s)\tilde\Psi(x,Y^1_s,\mu_s,v)\tilde\Psi(x,Y^2_s,\mu_s,v)\nu_1(dv)\mu_s(dx)ds.
\end{align*}

As we are interested in finding five white noises, we need to have five local martingales. This is why we introduce artificially $Y^i_t:=0$ for $3\leq i\leq 5.$ The rest of the proof is then an immediate consequence of Theorem~III-10 and Theorem~III-6 of \cite{el_karoui_martingale_1990}. More precisely, Theorem~III-10 implies the existence of five orthogonal martingale measures $M^i$ ($1\leq i\leq 5$) on $\r_+\times\r\times\r$ with intensity $dt\cdot \mu_t(dx)\cdot \nu_1(dv)$ such that
\begin{align*}
dY^1_t=&b(Y^1_t,\mu_t)dt+\int_\r\int_\r\sigma(Y^1_t,\mu_t)dM^1(t,x,v)\\
&+\int_\r\int_\r\sqrt{f(x,\mu_t)}\tilde\Psi(x,Y^1_t,\mu_t,v)dM^5(t,x,v)\\
&+\int_\r\int_\r\sqrt{f(x,\mu_t)}\kappa(x,Y^1_t,\mu_t)dM^3(t,p,v),\\
dY^2_t=&b(Y^2_t,\mu_t)dt+\int_\r\int_\r\sigma(Y^2_t,\mu_t)dM^2(t,x,v)\\
&+\int_\r\int_\r\sqrt{f(x,\mu_t)}\tilde\Psi(x,Y^2_t,\mu_t,v)dM^5(t,x,v)\\
&+\int_\r\int_\r\sqrt{f(x,\mu_t)}\kappa(x,Y^2_t,\mu_t)dM^4(t,x,v).
\end{align*}
Then, $\beta^i_t := M^i([0,t]\times\r\times\r)$ ($i=1,2$) are standard one-dimensional Brownian motions, and Theorem~III-6 of \cite{el_karoui_martingale_1990} allows us to write $M^i$ ($3\leq i\leq 5$) as
\begin{align*}
M^5([0,t]\times A\times B)=&\int_0^t\un_A(F^{-1}_s(p))\un_B(v)dW(s,p,v)\\
M^3([0,t]\times A\times B)=&\int_0^t\un_A(F^{-1}_s(p))dW^1(s,p)\\
M^4([0,t]\times A\times B)=&\int_0^t\un_A(F^{-1}_s(p))dW^2(s,p),
\end{align*}
where $W^1,W^2,W$ are white noises with respective intensities $dt\cdot dp,$ $dt\cdot dp$ and $dt\cdot dp\cdot\nu_1(dv)$.
\end{proof}

We now prove a key result for the proof of our main results. Recall that the martingale problem $(\mathcal{M})$ is given by definition \eqref{def:martpro}.
\begin{thm}
\label{convergingsubsequence}
Grant Assumptions~\ref{XNiwellposed},~\ref{barXiwellposed} and~\ref{hyppsi}. Then the law of every limit in distribution  of the sequence $(\mu^N)$ is solution of the martingale problem $(\mathcal{M})$.
\end{thm}

\begin{proof}
Let $\mu$ be the limit in distribution of some subsequence of $(\mu^N)$ and let $ Q = Q_\mu $ be its law. In the following, we still note this subsequence $(\mu^N).$ Firstly, we clearly have that $\mu^0=\nu^0$. For
 $0\leq s_1\leq ...\leq s_k\leq s;\;   \psi_1,...,\psi_k\in C_b(\mathcal{P}(\r)); \; \phi_1,...,\phi_k\in C_b(\r^2);\; \phi\in {C^3_b(\r^2)}$ define the following functional on $\mathcal{P} (D(\r_+,\r))$ :
\begin{equation*}
F(\mu):=
\psi_1(\mu_{s_1})...\psi_k(\mu_{s_k})\int_{D(\r_+,\r)^2}\mu\otimes\mu(d\gamma)\phi_1(\gamma_{s_1})\ldots\phi_k(\gamma_{s_k})
\ll[M_t^{\phi}(\mu,\gamma)-M_s^{\phi}(\mu,\gamma)\rr],
\end{equation*}
where $(M_t^{\phi}(\mu,\gamma))_t$ is given by~\eqref{M}.
To show that $(M_t^{\phi}(\mu,\gamma))_{t\geq 0}$ is a $(P_{Q_{\mu}}, (\G_t)_{t})$ martingale, we have to show that 
$$\int_{\mathcal{P} (D(\r_+,\r))} F(m)Q_{\mu}(dm)=\esp{F(\mu)}=0,$$
where the expectation $ \esp{ \cdot}$  is taken with respect to $ P_{Q_\mu} = P_Q.$ Note that the first equality is just the transfer formula, and hence the expectation is taken on the probability space where $\mu$ is defined.

{\it Step~1.} We show that  
\begin{equation}\label{eq:good}
\esp{ F( \mu^N) } \to \esp{  F( \mu)  } 
\end{equation} as $N \to \infty.$ This statement is not immediately clear since the functional $F$ is not continuous. The main difficulty comes from the fact that $\mu^N$ converges to $\mu$ in distribution for the topology of the weak convergence, but the terms appearing in the function $F$ require the convergence of $\mu^N_t$ to $\mu_t$ for the topology of the metric $W_1.$

The proof \eqref{eq:good} is actually quite technical, therefore we postpone it to the Appendix in Lemma~\ref{suitethm43}. \\

{\it Step 2.} In this step we show that $\esp{F(\mu)}$ is equal to $0.$ Applying $F$ to $\mu^N$ gives
\begin{multline}\label {FmuN}
F(\mu^N)=\psi_1(\mu^N_{s_1})...\psi_k(\mu^N_{s_k})\frac{1}{N^2}\sum_{i,j=1}^N\phi_1(X^{N,i}_{s_1},X^{N,j}_{s_1})...\phi_k(X^{N,i}_{s_k},X^{N,j}_{s_k})\\
\Big[\phi(X^{N,i}_t,X^{N,j}_t)-\phi(X^{N,i}_s,X^{N,j}_s) \\
-\int_s^tb(X^{N,i}_r,\mu^N_r)\partial_{x^1}\phi(X^{N,i}_r,X^{N,j}_r)dr-\int_s^tb(X^{N,j}_r,\mu^N_r)\partial_{x^2}\phi(X^{N,i}_r,X^{N,j}_r)dr\\
-\frac12\int_s^t\sigma(X^{N,i}_r,\mu^N_r)^2\partial^2_{x^1} \phi (X^{N,i}_r,X^{N,j}_r)dr-\frac12\int_s^t\sigma(X^{N,j}_r,\mu^N_r)^2\partial^2_{x^2}\phi (X^{N,i}_r,X^{N,j}_r)dr\\
-\frac12\int_s^t\frac 1 N\sum_{k=1}^N f(X^{N,k}_r,\mu^N_r)\kappa(X^{N,k}_r,X^{N,i}_r,\mu^N_r)^2\partial^2_{x^1}\phi(X^{N,i}_r,X^{N,j}_r)dr\\
-\frac12\int_s^t\frac 1 N\sum_{k=1}^N f(X^{N,k}_r,\mu^N_r)\kappa(X^{N,k}_r,X^{N,j}_r,\mu^N_r)^2\partial^2_{x^2}\phi(X^{N,i}_r,X^{N,j}_r)dr\\
-\frac12\int_s^t\int_\r\sum_{k=1}^N f(X^{N,k}_r,\mu^N_r)\frac 1 N\sum_{h,l=1}^2\tilde\Psi(X^{N,k}_r,X^{N,i_h}_r,\mu^N_r,v)\tilde\Psi(X^{N,k}_r,X^{N,i_l}_r,\mu^N_r,v) \cdot \\
 \partial^2_{x^hx^l}\phi(X^{N,i}_r,X^{N,j}_r)\nu_1(dv)dr\Big],
\end{multline}
with $i_1 = i$ and $i_2 = j.$
Denote $\tilde\pi^k(dr,dz,du):=\pi(dr,dz,du)-dr\cdot dz\cdot \nu(du)$  the compensated version  of $\pi^k.$
For $(i,j)~\in ~\llbracket 1,\ldots, N \rrbracket^2,\; s<t$ let us define 
\begin{multline}\label{Mij}
M_{s,t}^{N,i,j}:=\int_s^t\sigma(X^{N,i}_r,\mu^N_r)\partial_{x^1}\phi(X^{N,i}_r,X^{N,j}_r)d\beta^i_r
+\int_s^t\sigma(X^{N,j}_r,\mu^N_r)\partial_{x^2}\phi(X^{N,i}_r,X^{N,j}_r)d\beta^j_r;
\end{multline}

\begin{multline}\label{Wij}
W^{N,i,j}_{s,t}:=\sum_{k=1, k \neq i , j}^N\int_{]s,t]\times\r_+\times E}\uno{z\leq f(X^{N,k}_{r-},\, \mu^N_{r-})}\\
\ll[\phi(X^{N,i}_{r-}+\frac{1}{\sqrt{N}}\Psi(X^{N,k}_{r-},X^{N,i}_{r-},\mu^N_{r-},u^k,u^i),\; X^{N,j}_{r-} 
+\frac{1}{\sqrt{N}}\Psi(X^{N,k}_{r-},X^{N,j}_{r-},\mu^N_{r-},u^k,u^j))\rr.\\
\ll.-\phi(X^{N,i}_{r-},X^{N,j}_{r-})\rr]\tilde \pi^k(dr,dz,du);
\end{multline}

\begin{multline}\label{deltaij}
\Delta^{N,i,j}_{s,t}:=\sum_{k=1, k \neq i, j }^N\int_s^t\int_E f(X^{N,k}_r,\mu^N_{r-})\\
\ll[\phi(X^{N,i}_{r-}+\frac{1}{\sqrt{N}}\Psi(X^{N,k}_{r-},X^{N,i}_{r-},\mu^N_{r-},u^k,u^i),\: X^{N,j}_{r-}
 +\frac{1}{\sqrt{N}}\Psi(X^{N,k}_{r-},X^{N,j}_{r-},\mu^N_{r-},u^k,u^j))\rr. \\
\ll. -\phi(X^{N,i}_{r-},X^{N,j}_{r-})\rr]\nu(du)dr
\end{multline}
and

\begin{multline}\label{taylor2}
\Gamma^{N,i,j}_{s,t}=\Delta^{N,i,j}_{s,t}
\\
- \sum_{l=1}^2\sum_{k=1, k \neq i, j }^N\int_s^t\int_Ef(X^{N,k}_r,\mu^N_r)\frac{1}{\sqrt{N}}\Psi(X^{N,k}_{r-},X^{N,i_l}_{r-},\mu^N_{r-},u^k,u^{i_l})\partial_{x^l}\phi(X^{N,i}_r,X^{N,j}_r)\nu(du)dr\\
-\sum_{h,l=1}^2\int_s^t\int_E\frac1N\sum_{k=1, k\neq i, j }^Nf(X^{N,k}_r,\mu^N_r)\Psi(X^{N,k}_r,X^{N,i_h}_r,\mu^N_r,u^k,u^{i_h})\Psi(X^{N,k}_r,X^{N,i_l}_r,\mu^N_r,u^k,u^{i_l})\\
\partial^2_{x^hx^l}\phi(X^{N,i}_r,X^{N,j}_r)\nu(du)dr,
\end{multline}
with again $i_1 = i$ and $i_2 = j.$

Applying Ito's formula, we have
\begin{multline}\label {Ito}
\phi(X^{N,i}_t,X^{N,j}_t)=\phi(X^{N,i}_s,X^{N,j}_s)
+M^{N,i,j}_{s,t}+W^{N,i,j}_{s,t}+\Delta^{N,i,j}_{s,t}\\
+\int_s^t b(X^{N,i}_r,\mu^N_r)\partial_{x^1}\phi(X^{N,i}_r,X^{N,j}_r)dr
+\int_s^t b(X^{N,j}_r,\mu^N_r)\partial_{x^2}\phi(X^{N,i}_r,X^{N,j}_r)dr\\
+\frac12\int_s^t\sigma(X^{N,i}_r,\mu^N_r)^2\partial^2_{x^1}\phi(X^{N,i}_r,X^{N,j}_r)dr
+\frac12\int_s^t\sigma(X^{N,j}_r,\mu^N_r)^2\partial^2_{x^2}\phi(X^{N,i}_r,X^{N,j}_r)dr.
\end{multline}

Note that thanks to Assumption \eqref {hyppsi}$\; i)$, the term in the second line of  \eqref {taylor2} is zero. Again, 
if $i\neq j,$ using the definitions \eqref{psitilde} and \eqref{kappacarre}
\begin{multline*}
\int_E\Psi(X^{N,k}_r,X^{N,i}_r,\mu^N_r,u^k,u^{i})\Psi(X^{N,k}_r,X^{N,j}_r,\mu^N_r,u^k,u^{j})\nu(du)\\
=\int_\r\tilde\Psi(X^{N,k}_r,X^{N,i}_r,\mu^N_r,v)\tilde\Psi(X^{N,r},X^{N,j},\mu^N_r,v)\nu_1(dv),
\end{multline*}
and
\begin{multline*}
\int_E\Psi(X^{N,k}_r,X^{N,i}_r,\mu^N_r,u^k,u^{i})^2\nu(du)=\kappa(X^{N,k}_r,X^{N,i}_r,\mu^N_r)^2+\int_\r\tilde\Psi(X^{N,k}_r,X^{N,i}_r,\mu^N_r,v)^2\nu_1(dv).
\end{multline*}

This implies that the last line of~\eqref{taylor2} is equal to the sum of three last lines of \eqref{FmuN},  up to an error term which is of order $ O ( t/N) $ if we include the terms $ k = i, j $ in \eqref{taylor2}.

As a consequence, plugging \eqref{Ito} in \eqref {FmuN}, using the definitions \eqref{Wij}, \eqref {Mij}, \eqref{taylor2} and the previous remark we obtain
\begin{equation*}
F(\mu^N)=\psi_1(\mu^N_{s_1})...\psi_k(\mu^N_{s_k})\frac{1}{N^2}\sum_{i,j=1}^N\phi_1(X^{N,i}_{s_1},X^{N,j}_{s_1})...\phi_k(X^{N,i}_{s_k},X^{N,j}_{s_k})\ll[M^{N,i,j}_{s,t}+W^{N,i,j}_{s,t}+\Gamma^{N,i,j}_{s,t}\rr] .
\end{equation*}
Using \eqref {Ito} we see that  $(M^{N,i,j}_{s,t}+ W^{N,i,j}_{s,t}),\; \; {t\geq s},$ is a martingale with respect to the filtration $(\F_t^{X^N})_{t\geq 0}$  with  $\F_t^{X^N}:=\sigma (X^{N,i}_u; X^{N,j}_u; \ s\leq u\leq t )$ on the space where $X^N$  is defined.
Hence, using that $\phi_{s_k}$ and $\psi_{s_k}$ are bounded, and that $\mu^N$ is $(\F_t^{X^N})_{t\geq 0}$  adapted,
\begin{multline}
\esp{F(\mu^N)}=\esp{\esp{F(\mu^N)|\F_s^{X^N}}}=\\
\esp{\psi_1(\mu^N_{s_1})...\psi_k(\mu^N_{s_k})\frac{1}{N^2}\sum_{i,j=1}^N\phi_1(X^{N,i}_{s_1},X^{N,j}_{s_1})...\phi_k(X^{N,i}_{s_k},X^{N,j}_{s_k})\esp{\Gamma^{N,i,j}_{s,t}|\F_s^{X,\mu}}}\\
\leq
\frac{C}{N^2}\sum_{i,j=1}^N\esp{|\Gamma^{N,i,j}_{s,t}|},
\end{multline}
implying that 
$$|\esp{F(\mu^N)}|\leq C\esp{|\Gamma^{N,1,2}_{s,t}|+\frac{|\Gamma^{N,1,1}_{s,t}|}{N}}.$$
Taylor-Lagrange's inequality gives for all $i\neq j,$
\begin{multline*}
\esp{|\Gamma^{N,i,j}_{s,t}|}\leq \\
 C\frac{1}{N\sqrt{N}}\sum_{k=1, k \neq i, j }^N\sum_{n=0}^3\int_s^t\int_E
\esp{\Psi(X^{N,k}_r,X^{N,i}_r,\mu^N_r,u^k,u^{i})^n\Psi(X^{N,k}_r,X^{N,j}_r,\mu^N_r,u^k,u^{j})^{3-n}}\nu(du)dr\\
\leq C\frac{1}{N\sqrt{N}}\sum_{k=1, k \neq i, j }^N\int_s^t\int_E\esp{|\Psi(X^{N,k}_r,X^{N,i}_r,\mu^N_r,u^k,u^{i})|^3+|\Psi(X^{N,k}_r,X^{N,j}_r,\mu^N_r,u^k,u^{j})|^3}\nu(du)dr\\
\leq C\frac{1}{\sqrt{N}},
\end{multline*}
and a similar result holds for $\Gamma^{N,1,1}_{s,t}.$ Consequently,
$$|\esp{F(\mu^N)}|\leq C N^{-1/2},$$
implying together with \eqref{eq:good} that 
$$\esp{F(\mu)}=\underset{N}{\lim} \, \esp{F(\mu^N)}=0.$$
\end{proof}

\subsection{Proof of Theorem~\ref{convergencemuN}}

The beginning of the proof is similar to the proof of Theorem~2.6 of \cite{erny_conditional_2020}. It mainly consists in applying Lemma~\ref{representationmartingale} and Theorem~\ref{convergingsubsequence}.

Let $\mu$ be the limit in distribution of some converging subsequence of $(\mu^N)$ (that we still note $(\mu^N)$). Then, by Proposition~(7.20) of Aldous, $\mu$ is the directing measure of an exchangeable system $(\bar Y^i)_{i\geq 1},$ and $(X^{N,i})_{1\leq i\leq N}$ converges in distribution to $(\bar Y^i)_{i\geq 1}.$

According to Theorem~\ref{convergingsubsequence} and 
Lemma~\ref{representationmartingale}  for every $i\neq j,$ there exist, on an extension, Brownian motions $\beta^{i,j,1},\beta^{i,j,2}$ and white noises $W^{i,j,1},W^{i,j,2},W^{i,j}$ with respective intensities $dt\cdot dp,$ $dt\cdot dp$ and $dt\cdot dp\cdot\nu_1(dv)$ all independent such that
\begin{align*}
d\bar Y^{i}_t=&b(\bar Y^{i}_t,\mu_t)dt+\sigma(\bar Y^{i}_t,\mu_t)d\beta^{i,j,1}_t\\
&+\int_0^1\int_\r\sqrt{f(F_s^{-1}(p),\mu_t)}\tilde\Psi(F_s^{-1}(p),\bar Y^{i}_t,\mu_t,v)dW^{i,j}(t,p,v)\\
&+\int_0^1\sqrt{f(F_s^{-1}(p),\mu_t)}\kappa(F_s^{-1}(p),\bar Y^{i}_t,\mu_t)dW^{i,j,1}(t,p),\\
d\bar Y^{j}_t=&b(\bar Y^{j}_t,\mu_t)dt+\sigma(\bar Y^{j}_t,\mu_t)d\beta^{i,j,2}_t\\
&+\int_0^p\int_\r\sqrt{f(F_s^{-1}(p),\mu_t)}\tilde\Psi(F_s^{-1}(p),\bar Y^{j}_t,\mu_t,v)dW^{i,j}(t,p,v)\\
&+\int_\r\sqrt{f(F_s^{-1}(p),\mu_t)}\kappa(F_s^{-1}(p),\bar Y^{j}_t,\mu_t)dW^{i,j,2}(t,p),
\end{align*}
where $F_s$ is the distribution function related to $\mu_s$, and $F_s^{-1}$ its generalized inverse. We can construct this extension in a global way such that it works for all couples $(i, j ) $ simultaneously. 

As the system $(\bar Y^i)_{i\geq 1}$ is exchangeable, we know that $W^{i,j,1}=W^{i},$ $W^{i,j,2}=:W^j,$ $W^{i,j}=:W,$ $\beta^{i,j,1}=:\beta^{i}$ and $\beta^{i,j,2}=:\beta^{j}.$

It remains to identify the structure of $ \mu_t $ appearing above as the one prescribed in \eqref{mut} as conditional law. This is what we are going to do now. To do so, we introduce the following auxiliary system

\begin{align*}
dZ^{N,i}_t=&b(\tilde Z^{N,i}_t,\mu^{Z,N}_t)dt+\sigma(Z^{N,i}_t,\mu^{Z,N}_t)d\beta^{i}_t\\
&+\int_0^1\int_\r\sqrt{f((F^{Z,N}_s)^{-1}(p),\mu^{Z,N}_t)}\tilde\Psi((F^{Z,N}_s)^{-1}(p),Z^{N,i}_t,\mu^{Z,N}_t,v)dW(t,p,v)\\
&+\int_0^1\sqrt{f((F^{Z,N}_s)^{-1}(p),\mu^{Z,N}_t)}\kappa((F^{Z,N}_s)^{-1}(p),Z^{N,i}_t,\mu^{Z,N}_t)dW^{i}(t,p),
\end{align*}
where $\mu^{Z,N}:=N^{-1}\sum_{i=1}^N\delta_{Z^{N,i}},$ $F^{Z,N}$ is the distribution function related to $\mu^{Z,N}.$ In addition, in the rest of the proof, let $\mu^X:=\mathcal{L}(\bar X^1|\W)$ and $\mu^Y$ be the directing measure of $(\bar Y^i)_{i\geq 1}$ that was denoted by $\mu$ so far in this proof.

With similar computations as in Theorem~\ref{existencesolution}, one can prove that, for all $t\geq 0,$
\begin{align*}
\esp{(Z^{N,i}_t-\bar X^i_t)^2}\leq& C(1+t)\int_0^t\esp{(Z^{N,i}_s-\bar X^i_s)^2}ds +C(1+t)\int_0^t\esp{W_1(\mu^X_s,\mu^{Z,N}_s)^2}ds.
\end{align*}

Besides, introducing $\mu^{X,N}:=N^{-1}\sum_{i=1}^N\delta_{\bar X^i}$, we have
\begin{align*}
\esp{W_1(\mu^X_s,\mu^{Z,N}_s)^2}\leq& 2\esp{W_1(\mu^X_s,\mu^{X,N}_s)^2}+2\esp{W_1(\mu^{X,N}_s,\mu^{Z,N}_s)^2}\\
\leq& 2\esp{W_1(\mu^X_s,\mu^{X,N}_s)^2} + 2\esp{|Z^{N,i}_s-\bar X^i_s|}^2\\
\leq& 2\esp{W_1(\mu^X_s,\mu^{X,N}_s)^2} + 2\esp{(Z^{N,i}_s-\bar X^i_s)^2}.
\end{align*}

As a consequence, we obtain the following inequality, for all $t\geq 0,$
$$\esp{(Z^{N,i}_t-\bar X^i_t)^2}\leq C(1+t)\int_0^t\esp{(Z^{N,i}_s-\bar X^i_s)^2}ds+C_t\int_0^t\esp{W_1(\mu^X_s,\mu^{X,N}_s)^2}ds,$$
with $C_t$ a constant depending on~$t$. Gr\"onwall's lemma now implies that, for all $t\geq 0,$
\begin{equation}\label{controlZX}
\esp{(Z^{N,i}_t-\bar X^i_t)^2}\leq C_t\int_0^t\esp{W_1(\mu^X_s,\mu^{X,N}_s)^2}ds.
\end{equation}
Now, let us prove that the expression above vanishes by dominated convergence. Note that $\mu^{X,N}_s$ converges weakly to $\mu^X_s$ a.s. by Glivenko-Cantelli's theorem (applying the theorem conditionally on $\W$) recalling that the variables $\bar X^j_s$ are conditionally i.i.d. given $\W$. And $\mu^{X,N}_s(|x|)$ converges to $\mu^X_s(|x|)$ a.s. as a consequence of the strong law of large numbers applied conditionally on $\W$ (once again because of the conditional independence property). Hence the characterization~$(i)$ of Definition~6.8 and Theorem~6.9 of \cite{villani_optimal_2008} implies that $W_1(\mu^{X,N}_s,\mu^X_s)$ vanishes a.s. as $N$ goes to infinity for every~$s\geq 0.$

We also have, by Jensen's inequality and the definition of $W^1$, for all $t\geq 0,$
\begin{multline*}
\esp{W_1(\mu^X_s,\mu^{X,N}_s)^4}\leq C\esp{\mu^X_s(x^4)} + C\esp{\mu^{X,N}_s(x^4)} = C\esp{(\bar X^1_s)^4} + \esp{\frac1N\sum_{j=1}^N (\bar X^j_s)^4}\\
\leq C\esp{(\bar X^1_s)^4}.
\end{multline*}
Consequently
$$\underset{0\leq s\leq t}{\sup}\esp{W_1(\mu^X_s,\mu^{X,N}_s)^4} \leq C \underset{0\leq s\leq t}{\sup}\esp{(\bar X^1_s)^4}<\infty,$$
recalling~\eqref{controlbarXmckeandiffu}.

So we have just proven that the sequence of variables $W_1(\mu^X_s,\mu^{X,N}_s)^2$ vanishes a.s. and for every $s\geq 0,$ and that this sequence is uniformly integrable on $\Omega\times[0,t]$ for any $t\geq 0.$ This implies, by dominated convergence, that
$$\int_0^t\esp{W_1(\mu^X_s,\mu^{X,N}_s)^2}ds\underset{N\rightarrow\infty}{\longrightarrow}0.$$
Recalling~\eqref{controlZX}, we have that
\begin{equation}\label{controlZX2}
\esp{(Z^{N,i}_t-\bar X^i_t)^2}\underset{N\rightarrow\infty}{\longrightarrow}0.
\end{equation}
With the same reasoning, we can prove that for all $t\geq 0,$
\begin{equation}\label{controlZY}
\esp{(Z^{N,i}_t-\bar Y^i_t)^2}\underset{N\rightarrow\infty}{\longrightarrow}0.
\end{equation}
Finally, using~\eqref{controlZX2} and~\eqref{controlZY}, we have that, for all $t\geq 0,$
\begin{equation*}
\esp{(\bar Y^i_t-\bar X^i_t)^2}=0.
\end{equation*}
This proves that the systems $(\bar X^i)_{i\geq 1}$ and $(\bar Y^i)_{i\geq 1}$ are equal. Hence, recalling that $\mu=:\mu^Y$ is a limit of $\mu^N$ and also the directing measure of $(Y^i)_{i\geq 1},$ it is also the directing measure of $(\bar X^i)_{i\geq 1},$ which is $\mu^X:=\mathcal{L}(\bar X^1|\W).$

\subsection{Proof of Theorem~\ref{convergenceXNi}}

The proof of Theorem~\ref{convergenceXNi} is now a direct consequence of Theorem~\ref{convergencemuN} and Proposition~(7.20) of \cite{aldous_exchangeability_1983}.

\section{Model of interacting populations}\label{sec:multipop}

The aim of this section is to generalize the previous model, considering $n$ populations instead of one. Within each population, the particles interact as in the previous model, and, in addition, there are interactions at the level of the populations. The chaoticity properties of this kind of model have been studied by \cite{graham_chaoticity_2008}. Two examples of such systems are given in Figure~\ref{graphmultipop}.

\begin{figure}[h]
\begin{tikzpicture}
\node[shape=circle,draw=black] (un) at (0,0) {$1$};
\node[shape=circle,draw=black] (deux) at (1,-1) {$2$};
\node[shape=circle,draw=black] (trois) at (2,-2) {$3$};
\node[shape=circle,draw=black] (quatre) at (1,-3) {$4$};
\node[shape=circle,draw=black] (cinq) at (0,-4) {$5$};
\node[shape=circle,draw=black] (six) at (-1,-3) {$6$};
\node[shape=circle,draw=black] (sept) at (-2,-2) {$7$};
\node[shape=circle,draw=black] (huit) at (-1,-1) {$8$};

\path[->] (un) edge (deux);
\path[->] (deux) edge (trois);
\path[->] (trois) edge (quatre);
\path[->] (quatre) edge (cinq);
\path[->] (cinq) edge (six);
\path[->] (six) edge (sept);
\path[->] (sept) edge (huit);
\path[->] (huit) edge (un);

\node[shape=circle,draw=black] (dun) at (7,0) {$1$};
\node[shape=circle,draw=black] (ddeux) at (5,-2) {$2$};
\node[shape=circle,draw=black] (dtrois) at (9,-2) {$3$};
\node[shape=circle,draw=black] (dquatre) at (4,-4) {$4$};
\node[shape=circle,draw=black] (dcinq) at (6,-4) {$5$};
\node[shape=circle,draw=black] (dsix) at (8,-4) {$6$};
\node[shape=circle,draw=black] (dsept) at (10,-4) {$7$};

\path[<-] (dun) edge (ddeux);
\path[<-] (dun) edge (dtrois);
\path[<-] (ddeux) edge (dquatre);
\path[<-] (ddeux) edge (dcinq);
\path[<-] (dtrois) edge (dsix);
\path[<-] (dtrois) edge (dsept);
\end{tikzpicture}
\caption{Two examples of interacting populations.}
\label{graphmultipop}
\end{figure}
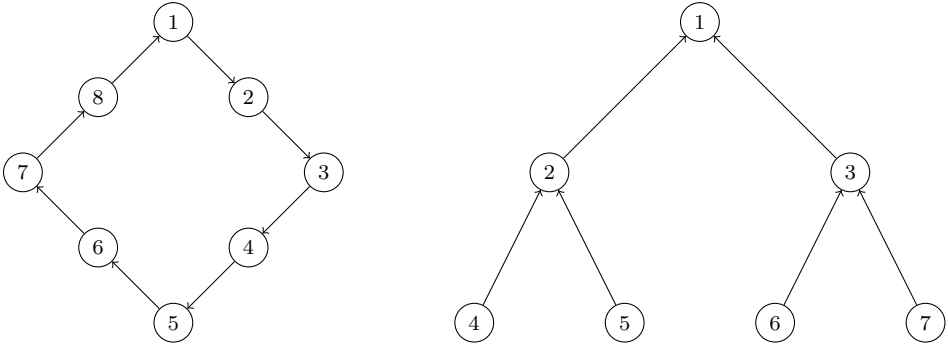

If we consider a number of $N$ particles, we note, for each $1\leq k\leq n,$ $N_k$ the number of particles of the $k-$th population. In particular $N=N_1+...+N_n.$ We assume that, for all $1\leq k\leq n,$ $N_k/N$ converges to some positive number, such that each population survives in the limit system.

For all $1\leq k\leq n,$ let $I(k)\subseteq\llbracket 1,n\rrbracket$ be the set of populations that are ``inputs" of the population~$k,$ that is, such that particles within these populations have a direct influence on those in population $k.$ The dynamic of the $N-$particle system $(X^{N,k,i})_{\substack{1\leq k\leq n\\1\leq i\leq N_k}}$ is governed by the following SDEs.

\begin{multline*}
dX^{N,k,i} = b^k\ll(X^{N,k,i}_t,\mu^{N,k}_t\rr)dt + \sigma^k\ll(X^{N,k,i}_t,\mu^{N,k}_t\rr)d\beta^{k,i}_t\\
+\sum_{l\in I(k)}\frac{1}{\sqrt{N_l}}\sum_{\substack{j=1\\(l,j)\neq (k,i)}}^{N_l}\int_{\r_+\times E^n}\Psi^{lk}(X^{N,l,j}_{t-},X^{N,k,i}_{t-},\mu^{N,l}_{t-},\mu^{N,k}_{t-},u^{l,j},u^{k,i})\uno{z\leq f^l(X^{N,l,j}_{t-},\mu^{N,l}_{t-})}d\pi^{l,j}(t,z,u).
\end{multline*}
In the above equation, 
$$ \mu^{N, k }_t=N_k^{-1}\sum_{j=1}^{N_k}\delta_{X^{N,k, j}_t},$$
$\pi^{l,j}$ ($1\leq l\leq n,j\geq 1$) are independent Poisson measures of intensity $dtdz\nu(du),$ where $\nu$ is a probability measure on $(\r^{\n^*})^n$ which is of the form 
$$\nu = (\nu^{1,1})^{\otimes \n^*}\otimes (\nu^{2,1})^{\otimes\n^*}\otimes...\otimes(\nu^{n,1})^{\otimes\n^*}.$$
The associated  limit system $(\bar X^{k,i})_{\substack{1\leq k\leq n\\i\geq 1}}$ is given by
\begin{align*}
d\bar X^{k,i} =& b^k\ll(\bar X^{k,i}_t,\bar \mu^{k}_t\rr)dt + \sigma^k\ll(\bar X^{k,i}_t,\bar\mu^{k}_t\rr)d\beta^{k,i}_t\\
&+\sum_{l\in I(k)}\int_\r\int_\r \sqrt{f^l(x,\bar\mu^l_t)}\tilde\Psi^{lk}(x,\bar X^{k,i}_t,\mu^l_t,\mu^k_t,v)dM^l(t,x,v)\\
&+\sum_{l\in I(k)}\int_\r \sqrt{f^l(x,\bar\mu^l_t)}\kappa^{lk}(x,\bar X^{k,i}_t,\mu^l_t,\mu^k_t)dM^{l,k,i}(t,x),
\end{align*}
with
$$\tilde\Psi^{lk}(x,y,m_1,m_2,v):= \int_\r \Psi^{lk}(x,y,m_1,m_2,v,w)d\nu^{k,1}(w),$$
\begin{multline*}
\kappa^{lk}(x,y,m_1,m_2)^2:=\int_{E^n} \Psi^{lk}(x,y,m_1,m_2,u^{l,1},u^{k,2})^2d\nu(u) - \int_\r \tilde\Psi^{lk}(x,y,m,v)^2d\nu^{l,1}(v)\\
= \int_{E^n} \Psi^{lk}(x,y,m_1,m_2,u^{l,1},u^{k,2})^2d\nu(u) - \int_{E^n} \Psi^{lk}(x,y,m,u^{l,1},u^{k,2})\Psi^{lk}(x,y,m,u^{l,1},u^{k,3})d\nu(u),
\end{multline*}
and
$$M^{l,k,i}_t(A)=\int_0^t\un_A((F^l_s)^{-1}(p))dW^{l,k,i}(s,p)\textrm{ and }M^l_t(A\times B) = \int_0^t\un_A((F^l_s)^{-1}(p))\un_B(v)dW^l(s,p,v).$$
In the above formulas, $\mu^k_t := \mathcal{L}\left(\bar X_t^{k,1}|\sigma\ll(\bigcup_{l\in I(k)}\mathcal{W}^l\rr)\right)$ and
 $(F^l_s)^{-1}$ is the generalized inverse of the function $F^l_s(x) := P(\bar X^{l,1}_s\leq x).$ Finally, $W^{l,k,i}$ and $W^l$ ($1\leq l\leq n,k\in I(l), i\geq 1$) are independent white noises of respective intensities $dsdp$ and $dsdp\nu^{l,1}(dv),$  and 
 $$\W^l_t:=\sigma\{\, W^l(]u,v]\times A\times B);  u<v\leq t;\  A\in \B([0,1]), B\in \B(\r) \, \};$$
$$\W^l:=\sigma\{W^l_t;\  t\geq 0\}.$$

Both previous systems are ``multi-exchangeable" in the sense that each population is ``internally exchangeable" as in Corollary~$(3.9)$ of \cite{aldous_exchangeability_1983}. With the same reasoning as in the proof of Theorem~\ref{existencesolution}, we can prove the existence of unique strong solutions $(X^{N,k,i})_{\substack{1\leq k\leq n\\1\leq i\leq N_k}}$ and $(\bar X^{k,i})_{\substack{1\leq k\leq n\\i\geq 1}}$
as well as the convergence of the $N-$particle system to the limit system:

\begin{thm}
\label{multipopgdsaut}
The following convergence in distribution in $\PP (D(\r_+,\r))^n$ holds true:
$$\ll(\mu^{N,1},\mu^{N,2},...,\mu^{N,n}\rr)\longrightarrow\ll(\bar\mu^{1},\bar\mu^{2},...,\bar\mu^{n}\rr),$$
as $N \to \infty .$ 
\end{thm}

Before giving a sketch of the proof of Theorem~\ref{multipopgdsaut}, we quickly state the following result.

\begin{prop}
Let $r\leq n$ and $1\leq k_1<...<k_r\leq n.$ If the sets $I(k_i)$ ($1\leq i\leq r$) are disjoint, then the random variables $\mu^{k_i}$ ($1\leq i\leq r$) are independent.
\end{prop}

\begin{proof}
For any $1\leq k\leq n,$ the system $(\bar X^{k,i})_{i\geq 1}$ is conditionally i.i.d. given $\sigma\ll(\bigcup_{l\in I(k)}\mathcal{W}^l\rr).$ So, by Lemma~$(2.12)$ of \cite{aldous_exchangeability_1983}, $\mu^k$ is $\sigma\ll(\bigcup_{l\in I(k)}\mathcal{W}^l\rr)-$measurable.
\end{proof}

\begin{rem}
In the two examples of Figure~\ref{graphmultipop}, all the variables $\mu^k$ ($1\leq k\leq n$) are independent.
\end{rem}

Coming back to Theorem~\ref{multipopgdsaut}, its proof  is similar to the proof of Theorem~\ref{convergencemuN}. The main argument relies on a generalization of the martingale problem discussed in Section \ref{section:mp}. Let us formulate it. Consider
$${\Omega' := \PP (D(\r_+,\r))^n\times \ll(D(\r_+,\r)^{2}\rr)^n,}$$
and write any atomic event $\omega'\in\Omega'$ as
$$\omega' = \ll(\mu^1,\mu^2,...,\mu^n,Y^{1,1},Y^{1,2},Y^{2,1},Y^{2,2},...,Y^{n,1},Y^{n,2}\rr)= ( \mu, Y ) ,$$
$ \mu = ( \mu^1, \ldots, \mu^n ) , Y = ( Y^1, \ldots, Y^n ).$

For $Q\in\mathcal{P} (\PP_1(D(\r_+,\r))^n),$ consider the law $P'$ on $\Omega'$ defined by
$$P'(A\times B_1\times ...\times B_n) = \int_{ \PP (D(\r_+,\r))^n} \un_A(m)m^1\otimes m^1(B_1)...m^n\otimes m^n(B_n)Q(dm),$$
with $A$ a Borel set of $\PP (D(\r_+,\r))^n$ and $B_1,...,B_n$ Borel sets of $D(\r_+,\r)^2.$

Then, we say that $Q$ is solution to our martingale problem if, for all $g\in C^2_b((\r^2)^n),$
$$g(Y_t) - g(Y_0) - \int_0^t\int_{\r^n}\int_{{\r^n}} Lg(Y_s,\mu_s,x,v)\mu_s^1\otimes...\otimes \mu_s^n(dx)\nu^{1,1}\otimes...\otimes\nu^{n,1}(dv)$$
is a martingale, where
\begin{multline*}
Lg(y,m,x,u) = \sum_{k=1}^n\sum_{i=1}^2 b^k(y^{k,i},m^k)\partial_{y^{k,i}}g(y) + \frac12 \sum_{k=1}^n\sum_{i=1}^2 \sigma^k(y^{k,i},m^k)^2\partial^2_{y^{k,i}}g(y)\\
+\frac12 \sum_{k=1}^n\sum_{l\in I(k)}\sum_{i=1}^2 f^l(x^{l},m^l)\kappa^{lk}(x^{l},y^{k,i},m^l,m^k)^2 \partial^2_{y^{k,i}}g(y)\\
+\frac12 \sum_{k_1,k_2=1}^n\sum_{i_1,i_2=1}^2\sum_{l\in I(k_1)\cap I(k_2)} f^l(x^{l},m^l) \tilde\Psi^{lk_1}(x^{l},y^{k_1,i_1},m^{l},m^{k_1},u^{l} ) \cdot \\
\tilde\Psi^{lk_2}(x^{l},y^{k_2,i_2},m^{l},m^{k_2},u^{l})\partial^2_{y^{k_1,i_1}y^{k_2,i_2}}g(y).
\end{multline*}

\begin{proof}[Sketch of proof of Theorem~\ref{multipopgdsaut}]
To prove the convergence in distribution of $(\mu^{N,1},...,\mu^{N,n})_N,$ we begin by proving its tightness. Following the same reasoning as in Section~\ref{tightnessmesuremartingale}, we can prove that, for each $1\leq k\leq n,$ the sequence $(\mu^{N,k})_N$ is tight on $\mathcal{P} (D(\r_+,\r)).$ Hence, the sequence $(\mu^{N,1},...,\mu^{N,n})_N$ on $\mathcal{P} (D(\r_+,\r))^n.$

Then  a generalization of Lemma~\ref{representationmartingale} allows to prove that the distribution of $(\bar\mu^1,...,\bar\mu^n)$ is the unique solution of the martingale problem defined above.

Finally, we can conclude the proof showing that the law of any limit of a converging subsequence of $(\mu^{N,1},...,\mu^{N,n})_N$ is solution to the martingale problem using similar computations as the one  in the proof of Theorem~\ref{convergingsubsequence}.
\end{proof}

\section{Appendix}

\subsection{A priori estimates}

\begin{lem}
\label{apriorimckeandiffu}
Grant Assumptions~\ref{XNiwellposed},~\ref{barXiwellposed} and~\ref{hyppsi}. For all $T>0,$
$$\underset{n\in\n^*}{\sup}\esp{\underset{t\leq T}{\sup} \ll|X^{N,1}_t\rr|^2}<\infty.$$
\end{lem}

\begin{proof}
Notice that 
$$\underset{0\leq s\leq t}{\sup} |X^{N,1}_s|\leq (X^{N,1}_0) + ||b||_\infty t + \underset{0\leq s\leq t}{\sup}\ll|\int_0^s \sigma(X^{N,1}_r,\mu^N_r)d\beta^1_r\rr| + \frac{1}{\sqrt{N}}\underset{0\leq s\leq t}{\sup}|M^N_s|,$$
where $M^N$ is the local martingale
$$M^N_t := \sum_{k=2}^N\int_{[0,t]\times\r_+\times E}\Psi(X^{N,k}_{s-},X^{N,1}_{s-},\mu^N_{s-},u^k,u^1)\uno{z\leq f(X^{N,k}_{s-},\mu^N_{s-})}d\pi^k(s,z,u).$$

Consequently, by Burkholder-Davis-Gundy's inequality and Assumption~\ref{hyppsi},
$$\esp{\underset{0\leq s\leq t}{\sup} |X^{N,1}_s|^2}\leq C + C||b||^2_{\infty}t^2 + ||\sigma||^2_\infty t + t||f||_\infty\frac{N-1}{N}\int_E\underset{x,y,m}{\sup}\Psi(x,y,m,u^1,u^2)^2 d\nu(u).$$
This proves the result.
\end{proof}

\subsection{Proof of \eqref{eq:good}}

\begin{lem}
\label{suitethm43}
Grant Assumptions~\ref{XNiwellposed},~\ref{barXiwellposed} and~\ref{hyppsi}. With the notation introduced in the proof of Theorem~\ref{convergingsubsequence}, we have
$$\esp{F(\mu^N)}\underset{N\rightarrow\infty}{\longrightarrow}\esp{F(\mu)}.$$
\end{lem}

\begin{proof}
Let us recall that $\mu^N$ denotes the empirical measure of $(X^{N,i})_{1\leq i\leq N}$ and that $\mu$ is the limit in distribution of (a subsequence of) $\mu^N.$

{\it Step~1.} We first show that almost surely, $ \mu $ is supported by continuous trajectories.  For that sake, we start showing that $ P^N := \esp{ \mu^N } = {\mathcal L} ( X^{N, 1 } ) $ is $C-$tight. This follows from Prop VI. 3.26 in \cite{jacod_limit_2003}, observing that 
$$ \lim_{N \to \infty} \esp{\sup_{s \le T } | \Delta X_s^{N, 1 } |^3 } = 0 , $$
which follows from our conditions on $ \psi.$ Indeed, writing $ \psi^* ( u^1, u^2 ) :=\sup_{x, y ,m } \psi ( x, y , m, u^1, u^2 ), $ we can stochastically upper bound
$$ \sup_{s \le T } | \Delta X_s^{N, 1 } |^3 \le \sup_{k \le K } | \psi^* ( U^{k, 1 }, U^{k, 2} ) |^3/N^{3/2} , $$
where $K \sim Poiss ( N T \|f\|_\infty ) $ is Poisson distributed with parameter $ N T \|f\|_\infty,$ and where $ (U^{k, 1 }, U^{k,2 } )_k $ is an i.i.d. sequence of $ \nu_1 \otimes \nu_1-$distributed random variables, independent of $K.$ The conclusion then follows from the fact that due to our Assumption \ref{hyppsi}, $ \esp{ | \psi^* ( U^{k, 1 }, U^{k, 2} ) |^3 } < \infty $ such that we can upper bound 
\begin{multline*}
 \esp{\sup_{k \le K } | \psi^* ( U^{k, 1 }, U^{k, 2} ) |^3/N^{3/2} } \le \esp{\frac{1}{N^{3/2}} \sum_{k=1}^K  | \psi^* ( U^{k, 1 }, U^{k, 2} ) |^3 }\\
  \le \frac{\esp{ | \psi^* ( U^{k, 1 }, U^{k, 2} ) |^3 }}{N^{3/2} } \esp{K} =\frac{\esp{ | \psi^* ( U^{k, 1 }, U^{k, 2} ) |^3 }}{N^{3/2} } N T \|f\|_\infty \to 0
\end{multline*}
as $ N \to 0.$ 
 
As a consequence of the above arguments, we know that $ \esp{ \mu ( \cdot ) } $ is supported by continuous trajectories. In particular, almost surely, $ \mu $ is also supported by continuous trajectories. Indeed, $\mu(C(\r_+,\r))$ is a r.v. taking values in $[0,1], $ and its expectation equals one. Thus $\mu(C(\r_+,\r))$ equals one a.s.

We now turn to the heart of this proof and show that $\esp{ F( \mu^N) } \to \esp{  F( \mu  } .$ The latter expression contains terms like 
$$\int_s^tb(Y^1_r,\mu _r)\partial_{x^1}\phi(Y^1_r,Y^2 _r)dr$$
for some bounded smooth function $ \phi .$  
However, by our assumptions, the continuity of $ m \mapsto b ( x, m) $ is expressed with respect to the Wasserstein $1-$distance. Yet, we only have information on the convergence of $\mu^N_r$ to $ \mu_r$  for the topology of the weak convergence.

In what follows we make use of Skorokhod's representation theorem and realize all random measures $ \mu^N $ and $\mu$ on an appropriate probability space such that we have almost sure convergence of these realizations (we do not change notation), that is, we know that almost surely, 
$$ \mu^N \to \mu $$
as $N \to \infty.$ (Recall that we have already chosen a subsequence in the beginning of the proof of Theorem~\ref{convergingsubsequence}). Since $\mu$ is almost surely supported by continuous trajectories, we also know that almost surely, $\mu_t^N \to \mu_t$ weakly for all $t$ (this is a consequence of Theorem~12.5.(i) of \cite{billingsley_convergence_1999}). 

{\it Step~2.} In a first time, let us prove that, a.s., for all $r,$ $\mu^N_r$ converges to $\mu_r$ for the metric $W_1$. Thus we need to show additionally that almost surely, for all $t \geq 0, $ $ \int |x| d \mu_t^N(x) \to \int |x| d \mu_t(x).$ 

To prove this last fact, it will be helpful to consider rather the convergence of the triplets $ ( \mu^N, X^{N, 1 }, \mu^N (|x|)).$  Since the sequence of laws of these triplets is tight as well (the tightness of $(\mu^N)_N$ and $(X^{N,1})_N$ have been stated in Section~\ref{tightnessmesuremartingale}, and the tightness of $(\mu^N(|x|)_N)$ is classical from Aldous' criterion since $\mu^N_t(|x|) = N^{-1}\sum_{k=1}^N |X^{N,k}_t|$), we may assume that, after having chosen another subsequence and then a convenient realization of this subsequence, we dispose of a sequence of random triplets such that almost surely, as $N \to \infty, $ 
$$  ( \mu^N, X^{N, 1 }, \mu^N ( |x| ) )  \to ( \mu, Y, A), $$
where $ A = ( A_t)_t$ is some process having c\`adl\`ag trajectories. In addition, it can be proven that the sequence $(\mu^N(|x|))_N$ is $C-$tight (for similar reasons as $(X^{N,1})_N$), hence $A$ has continuous trajectories.

Taking a bounded and continuous function $ \Phi: D( \r_+, \r) \to \r, $ we observe that, as $ N \to \infty, $  
$$ \esp{ \int_{D( \r_+, \r) } \Phi d \mu  } \leftarrow  \esp{ \int_{D( \r_+, \r) } \Phi d \mu^N  } = \esp{ \Phi( X^{N, 1} ) } \to \esp{ \Phi ( Y) } , $$
such that $ \esp{\mu} = {\mathcal L} ( Y). $

Notice that from the above follows that $Y$ is necessarily a continuous process, since $\esp{\mu}$ is supported by continuous trajectories. Notice also that for the moment we do not know if $ A = \mu ( |x|).$  

Using that $ \sup_N \esp{\sup_{t \le T }  |X_t^{N, 1 }|^2 } < \infty $ (see our a priori estimates Lemma~\ref{apriorimckeandiffu}), we deduce that the sequence 
$ (\sup_{ t \le T} |X_t^{N, 1 }|^{3/2} )_N$ is uniformly integrable. Therefore, $ \esp{\sup_{t \le T}  |X_t^{N, 1 } |^{3/2}} \to \esp{\sup_{ t \le T} | Y_t|^{3/2}} < \infty .$ 
In particular, we also have that  
$$\esp{\sup_{t\le T} \mu_t ( |x|^{3/2})} < \infty \quad \mbox{ and thus} \quad \sup_{t\le T} \mu_t ( |x|^{3/2}) < \infty \mbox{ almost surely,}$$ 
for all $T,$ since  
\begin{multline*}
 \esp{\sup_{t\le T} \mu_t ( |x|^{3/2})} = \esp{ \sup_{ t \le T } \int_{D(\r_+, \r)} | \gamma_t|^{3/2}  \mu (d \gamma) } \le  \esp{    \int_{D(\r_+, \r)} \sup_{ t \le T } | \gamma_t|^{3/2}  \mu (d \gamma) }  \\
 = \esp{ \sup_{ t \le T } | Y_t|^{3/2}} < \infty .
\end{multline*}

We know that, a.s., $\mu^N$ converges weakly to $\mu$ and $\mu(C(\r_+,\r))=1.$ Let us fix some $\omega\in\Omega$ for which the two previous properties hold. In the following, we omit this $\omega$ in the notation. Let $\varepsilon > 0, $ $t \le T$ and choose $M $ such that $ \int |x|\wedge M d \mu_t \geq \int |x| d \mu_t - \varepsilon. $ Then, as $N \to \infty, $ almost surely, 
$$ \int |x| d \mu_t^N \geq \int |x| \wedge M d \mu_t^N \to  \int |x| \wedge M d \mu_t .$$
Thus
$$ \liminf_N \int |x| d \mu_t^N \geq  \int |x| d \mu_t - \varepsilon , $$ 
such that
\begin{equation}\label{eq:liminf}
  \liminf_N  \int |x| d \mu_t^N \geq  \int |x| d \mu_t  .
\end{equation}
Fatou's lemma implies that 
$$ \esp{ \liminf_N \int |x| d \mu_t^N  } \leq \liminf_N  \esp{ \int |x| d \mu_t^N  }  = \liminf_N \esp{ |X_t ^{N, 1 } |} = \esp{ |Y_t|} = \esp{\int |x| d \mu_t}   .$$ 
Together with \eqref{eq:liminf} this implies that, almost surely,  
$$  \liminf_N \int |x| d \mu_t^N =   \int |x| d \mu_t  .$$ 
Finally, since $ \int |x| d\mu^N \to A $ and since $A$ is continuous, for all $t, $ 
$$   \liminf_N \int |x| d \mu_t^N = \limsup_N   \int |x| d \mu_t^N =   \int |x| d \mu_t  .$$   

This implies that almost surely, for all $t \geq 0, $ $ \int |x| d \mu_t^N(x) \to \int |x| d \mu_t(x)  = A_t < \infty.$ In particular, almost surely, for all $ t \geq 0, $ 
$$ W_1 ( \mu_t^N, \mu_t ) \to 0 $$ (see e.g. Theorem 6.9 of \cite{villani_optimal_2008}).

{\it Step~3.} Now we prove that $\esp{F(\mu^N)}$ converges to $\esp{F(\mu)},$ where we recall that
\begin{multline*}
F(\mu)=\psi_1(\mu_{s_1}) \cdot \ldots \cdot \psi_k(\mu_{s_k})\int_{D(\r_+,\r)^2}\mu\otimes\mu(d\gamma)\phi_1(\gamma_{s_1})\ldots\phi_k(\gamma_{s_k})\\
\ll[\phi(\gamma_t)-\phi(\gamma_s)-\int_s^t\int_\r\int_\r L\phi(\gamma_r,\mu_r,x,v)\nu_1(dv)\mu_r(dx)dr\rr],
\end{multline*}
where $\psi_i\in C_b(\mathcal{P}(\r)),\phi_i\in C_b(\r^2)$ ($1\leq i\leq k$) and $\phi\in C^3_b(\r^2).$ Let us recall some facts: by the boundedness of the functions $\psi_i$ ($1\leq i\leq k$) and our boundedness Assumption~\ref{hyppsi}, it is sufficient to prove the two following convergence:
\begin{align}
&\esp{|\psi_1(\mu^N_{s_1}) \cdot \ldots \cdot \psi_k(\mu^N_{s_k})-\psi_1(\mu_{s_1}) \cdot \ldots \cdot \psi_k(\mu_{s_k})|}\underset{N\rightarrow\infty}{\longrightarrow}0,\label{lespsi}\\
&\esp{|G(\mu^N)-G(\mu)|}\underset{N\rightarrow\infty}{\longrightarrow}0,\label{leg}
\end{align}
with
\begin{multline*}
G(\mu) := \int_{D(\r_+,\r)^2}\mu\otimes\mu(d\gamma)\phi_1(\gamma_{s_1})\ldots\phi_k(\gamma_{s_k})\\
\ll[\phi(\gamma_t)-\phi(\gamma_s)-\int_s^t\int_\r\int_\r L \phi (\gamma_r,\mu_r,x,v)\nu_1(dv)\mu_r(dx)dr\rr] .
\end{multline*}

Indeed, since the functions $\psi_i$ ($1\leq i\leq k$) and $G$ are bounded, we have
\begin{multline*}
\esp{|F(\mu^N)-F(\mu)|}\leq C \esp{|\psi_1(\mu^N_{s_1}) \cdot \ldots \cdot \psi_k(\mu^N_{s_k})-\psi_1(\mu_{s_1}) \cdot \ldots \cdot \psi_k(\mu_{s_k})|}  \\
+ C   \esp{|G(\mu^N)-G(\mu)|} .
\end{multline*}
The convergence~\eqref{lespsi} follows from dominated convergence and the fact that the function
$$m\in\mathcal{P}(D(\r_+,\r))\mapsto \psi_1(m_{s_1})...\psi_k(m_{s_k})\in\r$$
is bounded and continuous at~$\mu,$ since $\mu$ is supported by continuous trajectories.

To prove the convergence~\eqref{leg}, let us recall that we have already shown that
\begin{enumerate}
\item $\underset{N}{\sup}\underset{0\leq s\leq t}{\sup}\esp{\mu^N_s(|x|^{3/2}) }<\infty,$
\item $\underset{0\leq s\leq t}{\sup}\esp{\mu_t(|x|^{3/2})}<\infty,$
\item $\mu(C(\r_+,\r))=1~a.s.$
\item a.s. $\forall r,$ $\mu^N_r$ converges to $\mu_r$ for the metric $W_1,$
\item for all $x,x'\in\r,y,y'\in\r^2,m,m'\in\mathcal{P}_1(\r),v\in\r,$
$$|L\phi(y,m,x,v)-L\phi(y',m',x',v)|\leq C(v)(||y-y'||_1+|x-x'|+W_1(m,m')),$$
such that $\int_\r C(v)\nu_1(dv)<\infty,$
\item $$\int_\r \underset{x,y,m}{\sup} L\phi(y,m,x,v)\nu_1(dv)<\infty.$$
\end{enumerate}

In order to simplify the presentation, let us assume that the function $G$ is of the form 
$$G(\mu)=\int_{D^2}\mu \otimes \mu (d\gamma)\int_s^t\int_\r\int_\r L\phi(\gamma_r,\mu _r,x,v)\nu_1(dv)\mu _r(dx)dr.$$

Now, let us show that $\esp{|G(\mu^N)-G(\mu)|}$ vanishes as $N$ goes to infinity.

\begin{align*}
|G(\mu)-G(\mu^N)|\leq& \ll|G(\mu) - \int_{D^2}\mu^N\otimes\mu^N(d\gamma)\ll(\int_s^t\int_\r\int_\r L\phi(\gamma_r,\mu_r,x,v)\nu_1(dv)\mu_r(dx)dr\rr)\rr|\\
+&\ll|\int_{D^2}\mu^N\otimes\mu^N(d\gamma)\ll(\int_s^t\int_\r\int_\r L\phi(\gamma_r,\mu_r,x,v)\nu_1(dv)\mu_r(dx)dr\rr)\rr.\\
&\ll.- \int_{D^2}\mu^N\otimes\mu^N(d\gamma)\ll(\int_s^t\int_\r\int_\r L\phi(\gamma_r,\mu_r,x,v)\nu_1(dv)\mu^N_r(dx)dr\rr)\rr|\\
+&\ll|G(\mu^N) - \int_{D^2}\mu^N\otimes\mu^N(d\gamma)\ll(\int_s^t\int_\r\int_\r L\phi(\gamma_r,\mu_r,x,v)\nu_1(dv)\mu^N_r(dx)dr\rr)\rr|\\
&=:A_1+A_2+A_3.
\end{align*}

We first show that $A_1$ vanishes a.s. (this implies that $\esp{A_1}$ vanishes by dominated convergence).
$A_1$ is of the form
$$A_1 = \ll|\int_{D^2}\mu\otimes\mu(d\gamma) H(\gamma) - \int_{D^2}\mu^N\otimes\mu^N(d\gamma)H(\gamma)\rr|,$$
with
$$H : \gamma\in D^2 \mapsto  \int_s^t\int_\r\int_\r L\phi(\gamma_r,\mu_r,x,v)\nu_1(dv)\mu_r(dx)dr\in\r.$$

We just have to prove that $H$ is continuous and bounded. The boundedness is obvious, so let us verify the continuity.

Let $(\gamma^n)_n$ converge to $\gamma$ in $D(\r_+,\r)^2$. We have
\begin{align*}
|H(\gamma) - H(\gamma^n)|\leq& \int_s^t\int_\r\int_\r|H(\gamma_r,\mu_r,x,v) - H(\gamma^n_r,\mu_r,x,v)|\nu_1(dv)\mu_r(dx)dr\\
\leq& \int_s^t\int_\r\int_\r C(v) ||\gamma_r-\gamma^n_r||_1 \nu_1(dv)\mu_r(dx)dr\\
\leq& C\int_s^t ||\gamma_r-\gamma^n_r||_1dr,
\end{align*}
which vanishes by dominated convergence: the integrand vanishes at every continuity point~$r$ of $\gamma$ (whence for a.e. $r$), and, for $n$ big enough, $\sup_{r\leq t}||\gamma^n_r||_1 \leq 2\sup_{r\leq t}||\gamma_r||_1.$

Now we show that $\esp{A_2}$ vanishes. We have
\begin{multline*}
A_2\leq \int_{D^2}\mu^N\otimes\mu^N(d\gamma)\\
\ll(\int_s^t\ll|\int_\r\int_\r L\phi(\gamma_r,\mu_r,x,v)\nu_1(dv)\mu_r(dx) - \int_\r\int_\r L\phi(\gamma_r,\mu_r,x,v)\nu_1(dv)\mu^N_r(dx)\rr|dr\rr).
\end{multline*}

Since the function $x\in\r\mapsto\int_\r L\phi(\gamma_r,\mu_r,x,v)\nu_1(dv)$ is Lipschitz continuous (with Lipschitz constant independent of $\gamma_r$ and $\mu_r$), we have, by Kantorovich-Rubinstein duality (see e.g. Remark~6.5 of \cite{villani_optimal_2008}),
$$A_2\leq C\int_{D^2}\mu^N\otimes\mu^N(d\gamma)\int_s^t W_1(\mu^N_r,\mu_r)dr=C\int_s^tW_1(\mu^N_r,\mu_r)dr.$$
Hence
$$\esp{A_2}\leq C\int_s^t\esp{W_1(\mu^N_r,\mu_r)}dr,$$
which vanishes by dominated convergence: the integrand vanishes thanks to {\it Step~2}, and the uniform integrability follows from the fact that
$$\underset{N}{\sup}\int_s^t\esp{W_1(\mu^N_r,\mu_r)^{3/2}}dr\leq C(t-s)\underset{N}{\sup}\underset{0\leq s\leq t}{\sup}\esp{\mu^N_s(|x|)^{3/2}}+C(t-s)\underset{0\leq s\leq t}{\sup}\esp{\mu_s(|x|)^{3/2}}.$$

We finally show that $\esp{A_3}$ vanishes.
\begin{align*}
A_3\leq& \int_{D^2}\mu^N\otimes\mu^N(d\gamma)\ll(\int_s^t\int_\r\int_\r \ll|L\phi(\gamma_r,\mu^N_r,x,v)-L\phi(\gamma_r,\mu_r,x,v)\rr|\nu_1(dv)\mu^N_r(dx)dr\rr)\\
\leq&\int_{D^2}\mu^N\otimes\mu^N(d\gamma)\ll(\int_s^t\int_\r\int_\r C(v)W_1(\mu^N_r,\mu_r)\nu_1(dv)\mu^N_r(dx)dr\rr). 
\end{align*}

Then,
$$\esp{A_3}\leq C\int_s^t\esp{W_1(\mu^n_r,\mu_r)}dr,$$
which vanishes for the same reasons as in the previous step where we have shown that $\esp{A_2}$ vanishes. 
\end{proof}

\bibliography{biblio27-2-2021}
\end{document}